\documentclass[reqno,12pt]{amsart}
\usepackage{a4wide}
\usepackage{amsmath}
\usepackage{amssymb}
\usepackage{amsthm}
\usepackage{color}
\usepackage{hyperref}
\usepackage{breakurl}
\usepackage{comment}

\numberwithin{equation}{section}

\newtheorem{theorem}{Theorem}

\newtheorem{proposition}[theorem]{Proposition}
\newtheorem{lemma}[theorem]{Lemma}

\theoremstyle{definition}

\theoremstyle{remark}
\newtheorem{remark}[theorem]{Remark}
\newtheorem*{remarkno}{Remark}

\def\po#1#2{(#1)_#2}
\def\fl#1{\left\lfloor#1\right\rfloor}
\def\cl#1{\left\lceil#1\right\rceil}
\def\po#1#2{(#1)_#2}
\def\coef#1{\left\langle#1\right\rangle}
\def\Res{\operatorname{Res}}

\def\al{\alpha}

\def\({\left(}
\def\){\right)}
\def\[{\left[}
\def\]{\right]}

\def\fl#1{\left\lfloor#1\right\rfloor}

\definecolor{blaugrau}{rgb}{0.796887, 0.789075, 0.871107}

\newcounter{mmacnt}
\def\restartmma{\setcounter{mmacnt}{0}}
\restartmma \catcode`|=\active
\def|#1|{\mathrm{#1}}
\catcode`|=12
\newenvironment{mma}{
 \par
 \catcode`|=\active
 \parskip=2pt\parindent=0pt 
 \small
 \def\In##1\\{%
   \def\linebreak{\hfill\break\null\qquad}%
   \refstepcounter{mmacnt}
   \ifnum\value{mmacnt}>9
   \hangindent=3em\hangafter=0\else
   \hangindent=2.8em\hangafter=0\fi
   \leavevmode
   \llap{\tiny\sffamily In[\arabic{mmacnt}]:=\kern.5em}%
   \mathversion{bold}\scriptsize$\tt\bf\displaystyle##1$\normalsize
   \mathversion{normal}\par
 }%
 \def\Print##1\\{%
   \def\linebreak{\hfill\break}%
   \hangindent=2.5em\hangafter=0
   \leavevmode\scriptsize ##1\par}%
 \def\Out##1\\{%
   \vspace*{-0.2cm}\def\linebreak{$\hfill\break\null\hfill$}%
   \kern\abovedisplayskip\par
   \ifnum\value{mmacnt}>9
   \hangindent=3.3em\hangafter=0\else
   \hangindent=2.9em\hangafter=0\fi
   \leavevmode
   \llap{\tiny\sffamily Out[\arabic{mmacnt}]=\kern.5em}
   \scriptsize$\displaystyle\tt##1$\normalsize\hfill\null\par
   \kern\belowdisplayskip\vspace*{-0.1cm}
 }%
 \def\Warning##1##2\\{%
   \def\linebreak{\hfill\break}%
   \hangindent=2.5em\hangafter=0
   \leavevmode
   {\scriptsize##1 : ##2}\par}%
}{%
 \par\smallskip
}

\newcommand{\LoadP}[1]{\fcolorbox{black}{blaugrau}{
\begin{minipage}[t]{13.4cm}
\footnotesize #1
\end{minipage}}}

\newcommand{\myOut}[1]{{\sffamily Out[#1]}}

\def\MLabel#1{{\refstepcounter{mmacnt}\label{#1}}\addtocounter{mmacnt}{-1}}

\newcommand{\MText}[1]{\textbf{\ttfamily#1}}

\begin{document}

\title{Evaluation of binomial double sums involving absolute values}

\author{C. Krattenthaler}

\address[C. Krattenthaler]{Fakult\"at f\"ur Mathematik, Universit\"at Wien,
Oskar-Mor\-gen\-stern-\break Platz~1, A-1090 Vienna, Austria.
WWW: \tt http://www.mat.univie.ac.at/\~{}kratt.}

\author{C. Schneider}

\address[Carsten Schneider]{Research Institute for Symbolic Computation\\
J. Kepler University Linz\\
A-4040 Linz, Austria}
\email{Carsten.Schneider@risc.jku.at}

\thanks{Research partially supported by the Austrian Science Foundation (FWF) 
grant SFB F50 (F5005-N15 and F5009-N15) in the framework of the
Special Research Program
``Algorithmic and Enumerative Combinatorics".}



\begin{abstract}
We show that double sums of the form
$$
\sum_{i,j=-n} ^{n} 
|i^sj^t(i^k-j^k)^\beta| \binom {2n} {n+i} \binom {2n} {n+j}
$$
can always be expressed in terms of 
a linear combination of just four functions, namely
$\binom {4n}{2n}$, ${\binom {2n}n}^2$, $4^n\binom {2n}n$, and $16^n$,
with coefficients that are rational in~$n$. We provide two different
proofs: one is algorithmic and uses the second author's computer
algebra package {\tt Sigma}; the second is based on complex contour
integrals. In many instances, these results are extended to double
sums of the above form where $\binom {2n}{n+j}$ is replaced by
$\binom {2m}{m+j}$ with independent parameter~$m$.
\end{abstract}

\maketitle

\section{Introduction}
Motivated by work in \cite{BOS} concerning the Hadamard maximal determinant
problem~\cite{Hadamard83}, Brent and Osborn \cite{BO} proved the
double sum evaluation
\begin{equation} \label{eq:BO} 
\sum_{i,j=-n} ^{n} 
|i^2-j^2| \binom {2n} {n+i} \binom {2n} {n+j}
=2n^2{\binom {2n}n}^2.
\end{equation}
It should be noted that the difficulty in evaluating this sum
lies in the appearance of the absolute value. Without the absolute
value, the summand would become antisymmetric in~$i$ and~$j$ so
that the sum would trivially vanish.
Together with Ohtsuka and Prodinger, they went on in \cite{BOOPAA}
(see \cite{BOOPAB} for the published version)
to consider more general double sums of the form
\begin{equation} \label{eq:sumallg} 
\sum_{i,j=-n} ^{n} 
|i^sj^t(i^k-j^k)^\beta| \binom {2n} {n+i} \binom {2n} {n+j},
\end{equation}
mostly for small positive integers $s,t,k,\beta$.
Again, without the absolute value, the summation would not pose
any particular problem since it
could be carried out separately in~$i$ and~$j$ by means of a relatively
straightforward application of the binomial theorem.
In several cases, they found explicit evaluations
of such sums --- sometimes with proof, sometimes
conjecturally. 

The purpose of the current paper is to provide a complete treatment
of double sums of the form \eqref{eq:sumallg} and of the more
general form
\begin{equation} \label{eq:sumallg2} 
\sum_{i,j}  
|i^sj^t(i^k-j^k)^\beta| \binom {2n} {n+i} \binom {2m} {m+j},
\end{equation} 
with an independent parameter~$m$.
More precisely, 
using the computer algebra package {\tt Sigma}~\cite{Schneider:07a},
we were led to the conjecture that these double sums of the form~\eqref{eq:sumallg} can always be
expressed in terms of a linear combination of just four functions, namely
$\binom {4n}{2n}$, ${\binom {2n}n}^2$, $4^n\binom {2n}n$, and $16^n$,
with coefficients that are rational in~$n$, while in many instances
double sums of the form \eqref{eq:sumallg2} can be
expressed in terms of a linear combination of the four functions
$\binom {2n+2m}{n+m}$, ${\binom {2n}n}{\binom {2m}m}$, 
$4^n\binom {2m}m$, and $4^m\binom {2n}n$,
with coefficients that are rational in~$n$ and~$m$. 
We demonstrate this
observation in Theorems~\ref{thm:2}--\ref{thm:3B}, in a much more
precise form. 

It is not difficult to see that
the problem of evaluation of double sums of the form
\eqref{eq:sumallg} and \eqref{eq:sumallg2} 
can be reduced to the evaluation of sums of the form
\begin{equation} \label{eq:sum2} 
\sum_{0\le i\le j } 
i^sj^t \binom {2n} {n+i} \binom {2m} {m+j}
\end{equation}
(and a few simpler {\it single} sums). See the proofs of
Theorems~\ref{thm:2}--\ref{thm:3B} in Section~\ref{sec:Gl}
and Remark~\ref{rem:1}(1).
We furthermore show (see the proofs of 
Propositions~\ref{prop:0} and \ref{Prop:0}
in Section~\ref{sec:main}, which may be considered as the
actual main result of the present paper) that 
for the evaluation of double sums of the form \eqref{eq:sum2} it
suffices to evaluate four
{\it fundamental\/} double sums, given in
Lemmas~\ref{lem:10}--\ref{lem:22} in Section~\ref{sec:fund}. 
While Lemmas~\ref{lem:20}--\ref{lem:22} are relatively easy to prove
by telescoping arguments (see the proofs in Section~\ref{sec:fund}),
the proof of Lemma~\ref{lem:10} is more challenging.
We provide two different proofs,
one using computer algebra, and one using complex contour integrals.
We believe that both proofs are of intrinsic interest. The algorithmic
proof is described in Section~\ref{sec:Sigma}. There, we explain that the
computer algebra package {\tt Sigma} can be used in a completely automatic
fashion to evaluate double sums of the form \eqref{eq:sum2}. In particular,
the reader can see how we empirically
discovered our main results in Sections~\ref{sec:main} and \ref{sec:Gl}.
The second proof, based on the power of complex integration, is
explained in Section~\ref{sec:aux}.

We close our paper by proving another conjecture from 
\cite[Conj.~3.1]{BOOPAA}, namely the inequality (see
Theorem~\ref{thm:1} in Section~\ref{sec:Ungl})
$$
\sum_{i,j}\left\vert j^2-i^2\right\vert
\binom {2n}{n+i}
\binom {2m}{m+j}\ge 2nm \binom {2n}n\binom {2m}m.
$$
We show moreover that equality holds if and only
if $m=n$, in which case the evaluation \eqref{eq:BO} applies.
Although Lemmas~\ref{lem:10}--\ref{lem:22} would provide a good
starting point for a proof of the inequality, we prefer 
to use a more direct approach, involving an application of
Gosper's algorithm \cite{GospAB} at a crucial point.

We wish to point out that Bostan, Lairez and Salvy \cite{BoLSAA}
have developed an algorithmic approach --- based on contour integrals
--- that is capable of automatically finding a recurrence for the
double sum \eqref{eq:sumallg} for any particular choice of
$s,t,k,\beta$, and, thus, is able to establish an evaluation 
of such a sum (such as \eqref{eq:BO}, for example) once the right-hand
side is found. 

Our final remark is that some of the double sums \eqref{eq:sumallg}
and \eqref{eq:sumallg2} 
can be embedded into
infinite families of multidimensional sums that still allow for
closed form evaluations, see~\cite{BrKWAA}.

\section{The fundamental lemmas}
\label{sec:fund}

In this section, we state the summation identities which form the
basis of the evaluation of double sums of the form \eqref{eq:sum2}
(and, thus, of double sums of the form \eqref{eq:sumallg} and
\eqref{eq:sumallg2}).
As it turns out, Lemmas~\ref{lem:20}--\ref{lem:22} are very easy
to prove since at least one summation of the double sum can be
put in telescoping form, see the proofs below. Lemma~\ref{lem:10} is much
more subtle.
We provide two different proofs, the first being algorithmic ---
see Section~\ref{sec:Sigma}, the second making use of complex 
integration --- see Section~\ref{sec:aux}.

\begin{lemma} \label{lem:10}
For all non-negative integers $n$ and $m$, we have
\begin{multline} \label{eq:neu1} 
\sum_{0\le i\le j}
\binom {2n}{n+i}
\binom {2m}{m+j}
=
2^{2n+2m-3}
+\frac {1} {4}\binom {2n+2m} {n+m}
+\frac {1} {2}\binom {2n}n\binom {2m}m\\
+2^{2m-2}\binom {2n}n
-
\frac {1} {8}
\sum_{\ell=0} ^{n-m}
\binom {2n-2\ell}{n-\ell}
\binom {2m+2\ell} {m+\ell},
\end{multline}
where the sum on the right-hand side 
has to be interpreted as explained in Lemma~\ref{lem:13}.
\end{lemma}

\begin{lemma} \label{lem:20}
For all non-negative integers $n$ and $m$, we have
\begin{equation} \label{eq:neu10} 
\sum_{0\le i\le j}
i\,\binom {2n}{n+i}
\binom {2m}{m+j}
=
-\frac {n} {4}\binom {2n+2m} {n+m}
+n\,2^{2m-2}\binom {2n}n
+\frac {nm} {4(n+m)}\binom {2n}n\binom {2m}m.
\end{equation}
\end{lemma}

\begin{lemma} \label{lem:21}
For all non-negative integers $n$ and $m$, we have
\begin{equation} \label{eq:neu11} 
\sum_{0\le i\le j}
j\,\binom {2n}{n+i}
\binom {2m}{m+j}
=
\frac {m} {4}\binom {2n+2m} {n+m}
+\frac {m(m+2n)} {4(n+m)}\binom {2n}n\binom {2m}m.
\end{equation}
\end{lemma}

\begin{lemma} \label{lem:22}
For all non-negative integers $n$ and $m$, we have
\begin{equation} \label{eq:neu12} 
\sum_{0\le i\le j}
i\,j\,\binom {2n}{n+i}
\binom {2m}{m+j}
=
\frac {mn} {2(n+m)}\binom {2n+2m-2} {n+m-1}
+\frac {nm^2} {4(n+m)}\binom {2n}n\binom {2m}m.
\end{equation}
\end{lemma}

\begin{proof}[Proof of Lemma~\ref{lem:20}]
We have\footnote{\label{foot:1}The informed reader will have guessed that the
telescoping form of the summand was discovered by using Gosper's
algorithm \cite{GospAB} (see also \cite{PeWZAA}). 
The particular implementation that we applied is the
one due to Paule and Schorn \cite{PaScAA}.}
$$
i\binom {2n}{n+i}=\frac {n+i} {2}\binom {2n}{n+i}-
\frac {n+i+1} {2}\binom {2n}{n+i+1}.
$$
Thus, we obtain
\begin{align*}
\sum_{0\le i\le j}
i\,\binom {2n}{n+i}
\binom {2m}{m+j}
&=
\frac {1} {2}\sum_{j\ge0}
\left(n\binom {2n}{n}-
(n+j+1)\binom {2n}{n+j+1}\right)
\binom {2m}{m+j}\\
&=
\frac {n} {2}\binom {2n}{n}\sum_{j\ge0}
\binom {2m}{m+j}-
\frac {1} {2}\sum_{j\ge0}
(n-j)\binom {2n}{n+j}
\binom {2m}{m+j}.
\end{align*}
The first sum is, essentially, one half of a binomial theorem,
$$
\sum_{j\ge0}\binom {2m}{m+j}
=\frac {1} {2}\binom {2m}m+2^{2m-1}.
$$
In order to evaluate the second sum, we observe that\footnote{For the
finding of the telescoping form of the sum over $j\ge0$ below
see Footnote~\ref{foot:1}.}
\begin{align*}
\sum_{j\ge0}
(n-j)&\binom {2n}{n+j}
\binom {2m}{m+j}
=
n\sum_{j\ge0}
\binom {2n}{n+j}
\binom {2m}{m+j}
-
\sum_{j\ge0}
j\binom {2n}{n+j}
\binom {2m}{m+j}\\
&=
\frac {n} {2}\sum_{j=-\infty}^\infty
\binom {2n}{n+j}
\binom {2m}{m+j}
+\frac {n} {2}
\binom {2n}{n}
\binom {2m}{m}\\
&\kern1cm
-
\sum_{j\ge0}
\left(
\frac {(n+j)(m+j)} {2(m+n)}
\binom {2n}{n+j}
\binom {2m}{m+j}
\right.\\
&\kern3cm
-
\left.
\frac {(n+j+1)(m+j+1)} {2(m+n)}
\binom {2n}{n+j+1}
\binom {2m}{m+j+1}
\right)\\
&=
\frac {n} {2}\sum_{j=-\infty}^\infty
\binom {2n}{n+j}
\binom {2m}{m-j}
+\frac {n} {2}
\binom {2n}{n}
\binom {2m}{m}
-
\frac {nm} {2(m+n)}
\binom {2n}{n}
\binom {2m}{m}.
\end{align*}
The sum in the last line can be evaluated by means of the
Chu--Vandermonde summation formula (cf.\ \cite[Sec.~5.1,
  (5.27)]{GrKPAA}). Substitution of these findings and
little simplification then leads to the right-hand side of
\eqref{eq:neu10}.
\end{proof}

\begin{proof}[Proof of Lemma~\ref{lem:21}]
We have
\begin{equation} \label{eq:tele} 
j\binom {2m}{m+j}=\frac {m+j} {2}\binom {2m}{m+j}-
\frac {m+j+1} {2}\binom {2m}{m+j+1}.
\end{equation}
Thus, we obtain
\begin{align*}
\sum_{0\le i\le j}
j\,\binom {2n}{n+i}
\binom {2m}{m+j}
&=
\frac {1} {2}\sum_{i\ge0}
\binom {2n}{n+i}
(m+i)\binom {2m}{m+i}\\
&=
\frac {m} {2}\sum_{i\ge0}
\binom {2n}{n+i}
\binom {2m}{m+i}
+
\frac {1} {2}\sum_{i\ge0}
i\binom {2n}{n+i}
\binom {2m}{m+i}.
\end{align*}
We have evaluated the same sums in the previous proof.
We leave it to the reader to fill in the details in order to arrive
at the right-hand side of \eqref{eq:neu11}.
\end{proof}

\begin{proof}[Proof of Lemma~\ref{lem:22}]
Using \eqref{eq:tele}, we have
\begin{align*}
\sum_{0\le i\le j}
i\,j\,&\binom {2n}{n+i}
\binom {2m}{m+j}
=
\frac {1} {2}\sum_{i\ge0}
i\binom {2n}{n+i}
(m+i)\binom {2m}{m+i}\\
&=
\frac {1} {2}\sum_{i\ge0}
(n+i)\binom {2n}{n+i}
(m+i)\binom {2m}{m+i}\\
&\kern2cm
-
\frac {n} {2}\sum_{i\ge0}
\binom {2n}{n+i}
(m+i)\binom {2m}{m+i}\\
&=
2nm
\sum_{i\ge0}
\binom {2n-1}{n+i-1}
\binom {2m-1}{m+i-1}
-
\frac {n} {2}\sum_{i\ge0}
\binom {2n}{n+i}
(m+i)\binom {2m}{m+i}.
\end{align*}
We have evaluated the second sum in the previous proof.
In order to evaluate the first sum, we do the substitution
$i\to-i+1$ and obtain
\begin{align*}
\sum_{i\ge0}
\binom {2n-1}{n+i-1}
&\binom {2m-1}{m+i-1}
=
\frac {1} {2}
\sum_{i\ge0}
\binom {2n-1}{n+i-1}
\binom {2m-1}{m+i-1}\\
&\kern4cm
+
\frac {1} {2}
\sum_{i\le1}
\binom {2n-1}{n-i}
\binom {2m-1}{m-i}\\
&=
\frac {1} {2}\sum_{i=-\infty}^\infty
\binom {2n-1}{n+i-1}
\binom {2m-1}{m+i-1}\\
&\kern4cm
+
\frac {1} {2}
\binom {2n-1}{n-1}
\binom {2m-1}{m-1}
+
\frac {1} {2}
\binom {2n-1}{n}
\binom {2m-1}{m}\\
&=
\frac {1} {2}\sum_{i=-\infty}^\infty
\binom {2n-1}{n+i-1}
\binom {2m-1}{m-i}
+
\binom {2n-1}{n}
\binom {2m-1}{m}.
\end{align*}
Again, the sum can be evaluated by means of the
Chu--Vandermonde summation formula, and then
substitution of these findings and
little simplification leads to the right-hand side of
\eqref{eq:neu12}.
\end{proof}

\section{Proof of Lemma~\ref{lem:10}
using the computer algebra package {\tt Sigma}}
\label{sec:Sigma}

Here we show how Lemma~\ref{lem:10} can be established by
using the algorithmic tools provided by the summation package
\texttt{Sigma}~\cite{Schneider:07a} of the second author. 
Algorithmic proofs of Lemmas~\ref{lem:20}--\ref{lem:22} are much
simpler and could be obtained completely analogously. 

We seek an alternative representation of the double sum
\begin{equation}\label{Equ:DoubeSumOrg}
S(n,m)=\sum_ {0\le i\le j}
\binom {2 n} {n + i}
\binom {2 m} {m + j}
\end{equation}
for all non-negative integers $m,n$ 
with the following property: if one specialises $m$
(respectively\ $n$) to a non-negative integer or if one knows the distance
between $n$ and $m$, then the evaluation of the double sum should be
performed in a direct and simple fashion. In order to accomplish this
task, we utilise the summation package
\texttt{Sigma}~\cite{Schneider:07a}. 

The sum~\eqref{Equ:DoubeSumOrg}
can be rewritten in the form 
\begin{equation}\label{Equ:DoubleSumInput}
S(n,m)=\sum_{j=0}^{m}f(n,m,j)
\end{equation}
with
\begin{equation}\label{Equ:InnerSum}
f(n,m,j)=\binom{2 m}{j
+m
} 
\sum_{i=0}^j \binom{2 n}{i
+n
}.
\end{equation}
Given this sum representation we will exploit the following summation
spiral that is built into \texttt{Sigma}: 
\begin{enumerate}
 \item Calculate a linear recurrence in $m$ of order $d$
(for an appropriate positive integer~$d$) for the
   sum $S(n,m)$ by the creative telescoping paradigm; 
 \item solve the recurrence in terms of (indefinite) nested sums over
   hypergeometric products with respect to $m$ (the corresponding sequences
   are also called {\it d'Alembertian solutions}, see~\cite{PeWZAA}); 
 \item combine the solutions into an expression $\text{RHS}(n,m)$ such
   that $S(n,l)=\text{RHS}(n,l)$ holds for all $n$ and
   $l=0,1,\dots,d-1$.  
\end{enumerate}
Then this implies that $S(n,m)=\text{RHS}(n,m)$ holds for all
non-negative integers $m,n$.

\smallskip

\begin{remark} This summation engine can be considered as a
  generalisation of~\cite{PeWZAA} that works not only for
  hypergeometric products but for expressions in terms of nested sums
  over such hypergeometric products. It is based on a constructive
  summation theory of difference rings and
  fields~\cite{Schneider:16a,Schneider:16b} that enhances Karr's
  summation approach~\cite{Karr:81} in various directions.  
\end{remark}

In the following paragraphs, we assume that $m\leq n$. 
We activate \texttt{Sigma}'s summation spiral.

\medskip
\noindent{\sc Step 1.} Observe that our sum~\eqref{Equ:DoubleSumInput}
with summand given in~\eqref{Equ:InnerSum} 
is already in the right input form for
\texttt{Sigma}: the summation objects of~\eqref{Equ:InnerSum} are
given in terms of nested sums over hypergeometric products. More
precisely, let $\mathcal{S}_j$ denote the shift operator with respect
to~$j$, that is, $\mathcal{S}_jF(j):=F(j+1)$. Then,
if one applies this shift operator to 
the arising objects of $f(n,m,j)$, one can rewrite them again in their
non-shifted versions: 
\begin{equation}\label{Equ:ObjectRel}
\begin{split}
\mathcal{S}_j\binom{2 m}{j
+m
}&=\frac{m-j}{1
+j
+m
}\binom{2 m}{j
+m
},\\
\mathcal{S}_j\sum_{i=0}^j \binom{2 n}{i
+n
}&=
\sum_{i=0}^j \binom{2 n}{i
+n
}
+\frac{n-j
}{1
+j
+n
}\binom{2 n}{j
+n
}.
\end{split}
\end{equation}
With the help of these identities,
we can look straightforwardly for a linear
recurrence in the free integer parameter $m$ as follows. First, we
load \texttt{Sigma} into the computer algebra system {\sl Mathematica},

\begin{mma}
\In << Sigma.m \\
\Print \LoadP{Sigma - A summation package by Carsten Schneider
\copyright\ RISC-Linz}\\
\end{mma}

\noindent and enter our definite sum $S(n,m)$:

\begin{mma}
\In mySum = SigmaSum[Binomial[2 m, j + m] SigmaSum[Binomial[2 n, i +
      n], \{i, 0, j\}], \{j, 0, m\}]\\ 
\Out\sum_{j=0}^{m}\binom{2 m}{j
+m
} 
\sum_{i=0}^j \binom{2 n}{i
+n
}\\
\end{mma}

\noindent Then we compute a recurrence in $m$ by executing the function call

\begin{mma}
\In rec=GenerateRecurrence[mySum, m][[1]]\\ 
\Out SUM[m+1]-4 SUM[m]
==-\frac{1}{1
+m
+n
}\sum_{i=0}^m \binom{2 m}{i
+m
} \binom{2 n}{i
+n
}
+\frac{m n}{(m+1) (1
+m
+n
)}\binom{2 m}{m} \binom{2 n}{n}\\
\end{mma}

\noindent This means that $\texttt{SUM}[m]=S(n,m)(=\texttt{mySum})$ is
a solution of the output recurrence. But what is going on behind the scenes?
Roughly speaking,
Zeilberger's creative telescoping paradigm~\cite{PeWZAA} is carried
out in the setting of difference rings. More precisely, one tries to
compute a recurrence for the summand $f(n,m,j)$ of the form
\begin{multline}\label{Equ:SummandRec}
 c_0(n,m) f(n,m,j)+c_1(n,m) f(n,m+1,j)+\dots+c_d(n,m) f(n,m+d,j)\\
 =g(n,m,j+1)-g(n,m,j),
\end{multline}
for $d=0,1,2,\dots$. In our particular instance, \texttt{Sigma} is
successful for $d=1$ and delivers the solution 
$c_0(n,m)=-4$, $c_1(n,m)=1$, and 
\begin{multline} \label{eq:g-f}
g(n,m,j)=\frac{ (2 j-1)}{-1
        +j
        -m
        }\binom{2 m}{j
        +m
        } 
        \sum_{i=0}^j \binom{2 n}{i
        +n
        }\\
        +\frac{j
        -n
        }{1
        +m
        +n
        }\binom{2 m}{j
        +m
        } \binom{2 n}{j
        +n
        }
        +\frac{1}{-1
        -m
        -n
        }\sum_{i=0}^j \binom{2 m}{i
        +m
        } \binom{2 n}{i
        +n
        },
\end{multline}
which holds for all non-negative integers $j,m,n$ 
with $0\leq j\leq m\leq n$. The
correctness can be verified by substituting the right-hand side of
\eqref{Equ:InnerSum}
into~\eqref{Equ:SummandRec}, rewriting the summation objects in terms
of 
$\binom{2 m}{j
+m
}$ and $\sum_{i=0}^j \binom{2 n}{i
+n
}$ using the relations given in~\eqref{Equ:ObjectRel} and 
$\mathcal{S}_m\binom{2 m}{j
        +m
        }=
        \frac{2 (m+1) (2 m+1)}{(m-j+1) (1
        +j
        +m
        )}\binom{2 m}{j
        +m
        }
$, 
and applying simple rational function arithmetic. 
We recall that we assumed $m\le n$, and this restriction is indeed
essential for being allowed to use {\tt Sigma} in the described
setup. However, the above check reveals that the result is in fact
correct without any restriction on the relative sizes of $m$ and~$n$.

Finally,
by summing~\eqref{Equ:SummandRec} over $j$ from $0$ to $m$,
we obtain the linear recurrence 
\begin{multline*}
\sum_{j=0}^m f(n,m+1,j)-4\sum_{j=0}^m f(n,m,j)
=-
\sum_{j=0}^{m+1} \binom{2 n}{i
+n
}
+\frac{1}{-1
-m
-n
}\sum_{i=0}^m \binom{2 m}{i
+m
} \binom{2 n}{i
+n
}\\
+\frac{m n}{(m+1) (1
+m
+n
)}\binom{2 m}{m} \binom{2 n}{n}.
\end{multline*}
which, by the above remark, holds for all non-negative integers $m,n$.
As is straightforward to see, this is indeed equivalent to
\myOut3.

\medskip
\noindent{\sc Step~2.} 
We now apply our summation
toolbox to the definite sum $\sum_{i=0}^m \binom{2 m}{i 
+m
} \binom{2 n}{i
+n
}$ and obtain
\begin{equation} \label{eq:Chu1} 
\sum_{i=0}^m \binom{2 m}{m
+i
} \binom{2 n}{n
+i
}= \frac{1}{2} \binom{2 m}{m} \binom{2 n}{n}
+\frac{1}{2} \binom{2 m
+2 n
}{m
+n
}.
\end{equation}
Note that the calculations can be verified rigorously and as a
consequence we obtain a proof that the identity holds for all
non-negative integers $m,n$. 
Since we remain in this particular case purely in the hypergeometric
world, one could also use the classical toolbox described
in~\cite{PeWZAA}. Yet another (classical) proof consists in observing that
the sum on the left-hand side of \eqref{eq:Chu1} can be rewritten
as
\begin{multline*}
\frac {1} {2}
\Bigg(\sum_{i=0}^m \binom{2 m}{m
+i
} \binom{2 n}{
n-i
}
+
\sum_{i=0}^m \binom{2 m}{m-i
} \binom{2 n}{n+
i
}
\Bigg)\\
=
\frac {1} {2}
\Bigg(\sum_{i=0}^{2m} \binom{2 m}{i
} \binom{2 n}{
n+m-i
}
+\binom {2m}m\binom {2n}n
\Bigg),
\end{multline*}
and then evaluating the sum on the right-hand side 
by means of the Chu--Vandermonde summation formula.

As a consequence, we arrive at the linear
recurrence 

\begin{mma}\MLabel{MMA:RecMSmallerN}
\In rec=rec/.
\sum_{i=0}^m \binom{2 m}{i
+m
} \binom{2 n}{i
+n
}\to \frac{1}{2} \binom{2 m}{m} \binom{2 n}{n}
+\frac{1}{2} \binom{2 m
+2 n
}{m
+n
}\\
\Out SUM[m+1]-4 SUM[m]
==-\frac{\binom{2 m
+2 n
}{m
+n
}}{1
+m
+n
} \frac{1}{2}
+\frac{(-1
-m
+2 m n
)\binom{2 m}{m} \binom{2 n}{n}}{2(m+1) (1
+m
+n
)}
\\
\end{mma}

Now we can activate \texttt{Sigma}'s recurrence solver with
the function call 

\begin{mma}
\In recSol=SolveRecurrence[rec,SUM[m]]\\
\Out \{\{0, 2^{2 m}\},\{1, \frac{1}{4} \binom{2 m}{m} \binom{2 n}{n}
        +\frac{1}{4} \binom{2 m
        +2 n
        }{m
        +n
        }
        +2^{2 m} \binom{2 n}{n} \Big(
                -\frac{1}{4}
                +\frac{1}{4} n 
                \sum_{i=0}^m \frac{2^{-2 i} \binom{2 i}{i}}{i
                +n
                }
        \Big)\}\}\\
\end{mma}
\noindent This means that the first entry of the output is the
solution of the homogeneous version of the recurrence, and the second
entry is a solution of the recurrence itself. Hence, 
the general solution is
\begin{equation}\label{Equ:GeneralSol}
c\,2^{2 m}+ \frac{1}{4} \binom{2 m}{m} \binom{2 n}{n}
        +\frac{1}{4} \binom{2 m
        +2 n
        }{m
        +n
        }+2^{2 m} \binom{2 n}{n} \Bigg(
                -\frac{1}{4}
                +\frac{1}{4} n 
                \sum_{i=0}^m \frac{2^{-2 i} \binom{2 i}{i}}{i
                +n}\Bigg),
\end{equation}
where the constant $c$ (free of $m$) can be freely chosen. We note
that this solution can be easily verified by substituting it
into \MText{rec} computed in~\myOut{\ref{MMA:RecMSmallerN}} and using
the relations 
\begin{align*}
\mathcal{S}_m\binom{2 m}{m}&=\frac{2 (2 m+1) }{m+1}\binom{2 m}{m},\\
\mathcal{S}_m\binom{2 m
+2 n
}{m
+n
}&=\frac{2 (2 m
+2 n
+1
)}{m
+n
+1
}\binom{2 m
+2 n
}{m
+n
},\\
\mathcal{S}_m\sum_{i=0}^m \frac{2^{-2 i} \binom{2 i}{i}}{i
+n
}&=
\sum_{i=0}^m \frac{2^{-2 i} \binom{2 i}{i}}{i
+n
}
+\frac{2^{-2 m} (2 m+1)}{2(m+1) (1
+m
+n
)}\binom{2 m}{m}.
\end{align*}

\medskip
\noindent{\sc Step~3.} Looking at the initial value
$S(n,0)=\binom{2n}{n}$, we conclude that the specialisation 
$c=\frac{1}{2}\binom{2n}{n}$ in~\eqref{Equ:GeneralSol} equals 
$S(n,m)$ for all $n\geq0$ and $m=0$.  

\medskip
Summarising, we have found (together with a proof) the representation
\begin{equation}\label{Equ:SumRepInM}
S(n,m)=2^{2 m-2} \binom{2 n}{n} n 
\sum_{i=0}^m \frac{2^{-2 i} \binom{2 i}{i}}{i
+n
}
+2^{2 m-2} \binom{2 n}{n}
+\frac{1}{4} \binom{2 m}{m} \binom{2 n}{n}
+\frac{1}{4} \binom{2 m
+2 n
}{m
+n
},
\end{equation}
which holds for all non-negative integers $m,n$. 
This last calculation
step can be also carried out within \texttt{Sigma}, by making use of 
the function call

\begin{mma}
\In FindLinearCombination[recSol, \{0, \{\tbinom{2n}{n}\}\}, m, 1]\\
\Out 2^{2 m-2} \binom{2 n}{n} n 
\sum_{i=0}^m \frac{2^{-2 i} \binom{2 i}{i}}{i
+n
}
+2^{2 m-2} \binom{2 n}{n}
+\frac{1}{4} \binom{2 m}{m} \binom{2 n}{n}
+\frac{1}{4} \binom{2 m
+2 n
}{m
+n
}\\
\end{mma}

\medskip
Strictly speaking, the above derivations contained one ``human" 
(=~non-auto\-ma\-tic) step, namely at the point where we checked
\eqref{eq:g-f} and observed that this relation actually holds
without the restriction $m\le n$. For the algorithmic ``purist"
we point out that it is also possible to set up the problem
appropriately under the restriction $m>n$ (by splitting the double sum
$S(n,m)$ into two parts) so that {\tt Sigma} is applicable.
Not surprisingly, {\tt Sigma} finds \eqref{Equ:SumRepInM} again.

\medskip
In this article, we are particularly interested in the
evaluation of $S(n,m)$ if one fixes the distance $r=n-m\geq0$ (or
$r=m-n\geq0$).   
In order to find such a representation for the case $m\leq n$, we
manipulate the obtained sum  
\begin{equation}\label{Equ:SubSum1}
\sum_{i=0}^m \frac{2^{-2 i} \binom{2 i}{i}}{i
+n
}=\sum_{i=0}^m \frac{2^{-2 i} \binom{2 i}{i}}{i
+r+m
}:=T(m,r)
\end{equation}
in~\eqref{Equ:SumRepInM} further by applying once more \texttt{Sigma}'s
summation spiral (where $r$ takes over the role of $m$).

\medskip
\noindent{\sc Step~1.} Using \texttt{Sigma} (alternatively one
could use the  
Paule and Schorn implementation~\cite{PaScAA} of Zeilberger's
algorithm), we obtain the recurrence 
$$2 (m
+r
) T(m,r)
+(-1
-2 m
-2 r
) T(m,r+1)
=\frac{2^{-2 m} (2 m+1) \binom{2 m}{m}}{2 m
+r
+1
}.$$

\medskip
\noindent{\sc Step~2.} Using \texttt{Sigma}'s recurrence solver
we obtain the general solution 
$$
d\,\frac{2^{2 r}m \binom{2 m}{m}}{\binom{2 m
        +2 r
        }{m
        +r
        } (m
        +r
        )}
        + \frac{2^{-2 m} \binom{2 m}{m}}{m
        +r
        }
        -\frac{2^{-2 m+2 r} (4 m+1)\binom{2 m}{m}^2}{2\binom{2 m
        +2 r
        }{m
        +r
        } (m
        +r
        )}
        -\frac{2^{2 r-2 m}m \binom{2 m}{m}}{\binom{2 m
        +2 r
        }{m
        +r
        } (m
        +r
        )} 
        \sum_{i=0}^r \tfrac{2^{-2 i} \binom{2 m
        +2 i
        }{m
        +i
        }}{2 m
        +i
        },
$$
where the constant $d$ (free of $r$) can be freely chosen.

\medskip
\noindent{\sc Step~3.} Looking at the initial value 
$$T(m,0)=\sum_{i=0}^m \frac{2^{-2 i} \binom{2 i}{i}}{i
+m
}=\frac{2^{2 m-1}}{m \binom{2 m}{m}}
+\frac{2^{-2 m-1} \binom{2 m}{m}}{m},
$$
which we simplified by another round of \texttt{Sigma}'s summation
spiral, we conclude that we have to specialise $d$ to  
$$d=\frac{2^{2 m-1}}{m \binom{2 m}{m}}
+\frac{2^{-2 m-1} (4 m+1) \binom{2 m}{m}}{m}.$$
With this choice, we end up at the identity
$$T(m,r)=-\frac{2^{2 r-2 m} m\binom{2 m}{m}}{\binom{2 m
+2 r
}{m
+r
} (m
+r
)}  
\sum_{i=0}^r \frac{2^{-2 i} \binom{2 i
+2 m
}{i
+m
}}{i
+2 m
}
+\frac{2^{-2 m} \binom{2 m}{m}}{m
+r
}
+\frac{2^{2 m+2 r-1}}{\binom{2 m
+2 r
}{m
+r
} (m
+r
)},
$$
being valid for all non-negative integers $r,m$.
Finally, performing the substitution $r\to n-m$, we find the identity 
\begin{equation}\label{Equ:SumRepInDiff}
T(m,n-m)=-\frac{2^{2 n-4 m} \binom{2 m}{m}}{n \binom{2 n}{n}} m 
\sum_{i=0}^{n-m
} \frac{2^{-2 i} \binom{2 i
+2 m
}{i
+m
}}{i
+2 m
}
+\frac{2^{2 n-1}}{n \binom{2 n}{n}} 
+\frac{2^{-2 m} \binom{2 m}{m}}{n},
\end{equation}
which holds for all non-negative integers $n,m$ with $n\geq m$.
By substituting this result into~\eqref{Equ:SumRepInM}, we see that we
have discovered {\it and\/} proven that
\begin{multline} \label{eq:lem10-1}
S(n,m)=-2^{-2 m+2 n-2} \binom{2 m}{m} m 
\sum_{i=0}^{n-m
} \frac{2^{-2 i} \binom{2 i
+2 m
}{i
+m
}}{i
+2 m
}\\
+2^{2 m-2} \binom{2 n}{n}
+\frac{1}{2} \binom{2 m}{m} \binom{2 n}{n}
+\frac{1}{4} \binom{2 m
+2 n
}{m
+n
}
+2^{2 m+2 n-3},
\end{multline}
which is valid for all non-negative integers $n,m$ with $n\geq m$. 
In a similar fashion, if $m\ge n$, we obtain
\begin{multline} \label{eq:lem10-2}
S(n,m)=2^{2 m-2 n-2} \binom{2 n}{n} n 
\sum_{i=0}^{m
-n
} \frac{2^{-2 i} \binom{2 i
+2 n
}{i
+n
}}{i
+2 n
}\\
+2^{2 m-2} \binom{2 n}{n}
+\frac{1}{4} \binom{2 m}{m} \binom{2 n}{n}
+\frac{1}{4} \binom{2 m
+2 n
}{m
+n
}
+2^{2 m+2 n-3}.
\end{multline}

\medskip

We note that the interaction of the summation steps~1--3 is
carried out at various places in a recursive manner. In order to free
the user from all these mechanical but rather subtle calculation
steps, the additional package \texttt{EvaluateMultiSums}~\cite{Schneider:13a} 
has been developed recently. It coordinates all these
calculation steps cleverly and discovers identities as above
completely automatically whenever such a simplification in terms of
nested sums over hypergeometric products is possible. For instance,
after loading the package 

\begin{mma}
\In << EvaluateMultiSum.m \\
\Print \LoadP{EvaluateMultiSums by Carsten Schneider
\copyright\ RISC-Linz}\\
\end{mma}

\noindent we can transform the sum~\eqref{Equ:DoubeSumOrg} into the
desired form by executing the function call 

\begin{mma}
\In res=EvaluateMultiSum[\tbinom {2 n} {n + i}
\tbinom {2 m} {m + j},\{\{i,0,j\},\{j,0,m\}\},\{m,n\},\{0,0\},\{n,\infty\}]\\
\Out \frac{(2 n+1)2^{2 m-3} (2 n)!}{n^2 ((n-1)!)^2} 
\sum_{i=1}^m \frac{2^{-2 i} \binom{2 i}{i}}{1
+i
+n
}\newline
\hspace*{1cm}+\frac{(4 n+3)2^{2 m-3} (2 n)!}{n^2 (n+1) ((n-1)!)^2}
+\frac{(3
+4 m
+2 n
)\binom{2 m}{m} (2 n)!}{8n^2 (1
+m
+n
) ((n-1)!)^2} 
+\frac{(2 m
+2 n
)!}{4n^2 ((n-1)!)^2 \big(
        (n+1)_m\big)^2}\\
\end{mma}
\noindent Here, {\tt Sigma} uses the {\it Pochhammer symbol} $(\al)_m$
defined by
\begin{equation} \label{eq:Poch} 
(\alpha)_m  = 
\begin{cases} 
\alpha(\hbox{$\alpha+1$})(\alpha+2)\cdots (\alpha+m-1),&
\text{for
$m>0$,}\\
1,&
\text{for
$m=0$,}\\
1/(\alpha-1)(\hbox{$\alpha-2$})(\alpha-3)\cdots (\alpha+m),&
\text{for
$m<0$,}
\end{cases}
\end{equation}
which we shall also use later.
The parameters $m,n$ in the calculation above are bounded from below by $0,0$
and from above by $n,\infty$, respectively. If one prefers a
representation purely in terms of binomial coefficients, one may execute 
the following function calls: 

\begin{mma}
\In res=SigmaReduce[res, m, 
 Tower \to \{\tbinom{2 m}{m}, \tbinom{2 n + 2 m}{ n + m}\}];\\
\In res=SigmaReduce[res, n, 
 Tower \to \{\tbinom{2 n}{n}\}];\\
\Out 2^{2 m-3} (2 n+1)\binom{2 n}{n}  
\sum_{i=1}^m \frac{2^{-2 i} \binom{2 i}{i}}{1
+i
+n
}
+\frac{(4 n+3)2^{2 m-3} \binom{2 n}{n}}{n+1} 
+\frac{(3
+4 m
+2 n
)\binom{2 m}{m} \binom{2 n}{n}}{8(1
+m
+n)
} 
+\frac{1}{4} \binom{2 m
+2 n
}{m
+n
}\\
\end{mma}

\noindent If one rewrites the arising sum manually by means of
the function call below,
one finally ends up exactly at the result given
in~\eqref{Equ:SumRepInM}:

\begin{mma}
\In res=SigmaReduce[res,m,Tower\to\{\sum_{i=1}^m \tfrac{2^{-2 i}
    \tbinom{2 i}{i}}{i 
        +n
        }\}]\\
\Out 2^{2 m-2} \binom{2 n}{n} n 
\sum_{i=1}^m \frac{2^{-2 i} \binom{2 i}{i}}{i
+n
}
+2^{2 m-1} \binom{2 n}{n}
+\frac{1}{4} \binom{2 m}{m} \binom{2 n}{n}
+\frac{1}{4} \binom{2 m
+2 n
}{m
+n
}\\
\end{mma}

Analogously one can carry out these calculation steps to
calculate the simplification given
in~\eqref{Equ:SumRepInDiff} automatically. 

\medskip

Comparison with Lemma~\ref{lem:10} reveals that \eqref{eq:lem10-1} or
\eqref{eq:lem10-2} do not quite agree with the right-hand side of
\eqref{eq:neu1}. For example, in order to prove that
\eqref{eq:lem10-1} is equivalent with \eqref{eq:neu1}, we would have
to establish the identity
$$
\frac {1} {8}
\sum_{l=0}^{n-m
} \binom{2 m
+2 l
}{m
+l
} \binom{2 n
-2 l
}{n-l}
=
2^{-2 m+2 n-2} \binom{2 m}{m} m 
\sum_{i=0}^{n-m
} \frac{2^{-2 i} \binom{2 i
+2 m
}{i
+m
}}{i
+2 m
}.
$$
This can, of course, be routinely achieved by using the
Paule and Schorn implementation~\cite{PaScAA} of Zeilberger's
algorithm. Alternatively, we may use our {\tt Sigma}
summation technology again. Let
$$T'(n,m):=\sum_{l=0}^{n-m
} \binom{2 m
+2 l
}{m
+l
} \binom{2 n
-2 l
}{n-l}.$$
The above described summation spiral leads to
$$
T'(n,m)=-2^{2 m+1} n \binom{2 n}{n} 
\sum_{i=0}^m \frac{2^{-2 i} \binom{2 i}{i}}{i
+n
}
+2 \binom{2 m}{m} \binom{2 n}{n}
+2^{2 m+2 n}.$$
If this relation is substituted in \eqref{Equ:SumRepInM}, then we
arrive exactly at the assertion of Lemma~\ref{lem:10}.

Clearly, the case where $m\ge n$ can be treated in a similar fashion.
This finishes the algorithmic proof of Lemma~\ref{lem:10}. \qed

\medskip


\section{Proof of Lemma~\ref{lem:10}
using complex contour integrals}
\label{sec:aux}

In this section, we show how to prove
Lemma~\ref{lem:10} by making use of complex contour
integrals. Before we can embark on the proof of the lemma, we
need to establish several auxiliary evaluations of specific
contour integrals.

\begin{remarkno}
In order to avoid a confusion of the summation index~$i$ with the usual
short notation for $\sqrt{-1}$, throughout this section we write $\mathbf i$
for $\sqrt{-1}$.
\end{remarkno}

\begin{lemma} \label{lem:11}
For all non-negative integers $n$, we have
\begin{equation} \label{eq:neu2} 
\frac {1} {2\pi \mathbf{i}}\int _{\mathcal C} ^{}
\frac {dz} {z^{n+1}(1-z)^{n+1}}
\frac {1} {(1-2z)}=2^{2n},
\end{equation}
where $\mathcal C$ is a contour close to $0$, which encircles
$0$ once in the positive direction.
\end{lemma}

\begin{proof}
Let $I_1$ denote the expression on the left-hand side of
\eqref{eq:neu2}. We blow up the contour $\mathcal C$ so that
it is sent to infinity. While doing this, we must pass over
the poles $z=1/2$ and $z=1$ of the integrand. This must be
compensated by taking the residues at these points into account.
Since the integrand is of the order $O(z^{-2})$ as $\vert z\vert\to\infty$, 
the integral along the contour
near infinity vanishes. Thus, we obtain
\begin{align*} 
I_1&=
-\Res_{z=1/2}
\frac {1} {z^{n+1}(1-z)^{n+1}}
\frac {1} {(1-2z)}
-\Res_{z=1}
\frac {1} {z^{n+1}(1-z)^{n+1}}
\frac {1} {(1-2z)}\\
&=
2^{2n+1}
-\frac {1} {2\pi \mathbf{i}}\int _{\mathcal C} ^{}
\frac {1} {(1+z)^{n+1}(1-(1+z))^{n+1}}
\frac {1} {(1-2(1+z))}\,dz.
\end{align*}
As the substitution $z\to -z$ shows, the last integral is
identical with $I_1$. Thus, we have obtained an equation
for $I_1$, from which we easily get the claimed result.
\end{proof}

\begin{lemma} \label{lem:12}
For all non-negative integers $n$ and $m$, we have
\begin{equation} \label{eq:neu3} 
\frac {1} {(2\pi \mathbf{i})^2}\int _{\mathcal C_1} ^{}\int _{\mathcal C_2}
\frac {1} {(u-t)}
\frac {du} {u^{n+1}(1-u)^{n+1}}
\frac {dt} {t^{m}(1-t)^{m}}=
-\frac {1} {2}\binom {2n+2m}{n+m},
\end{equation}
where $\mathcal C_1$ and $\mathcal C_2$ are contours close to $0$,
which encircle $0$ once in the positive direction, and $C_2$ is entirely
in the interior of $C_1$.
\end{lemma}

\begin{proof}
We treat here the case where $n\ge m$. The other case can be disposed
of completely analogously.

Let $I_2$ denote the expression on the left-hand side of
\eqref{eq:neu3}. Clearly, interchange of $u$ and $t$ in the
integrand does not change $I_2$. In that case however, we must also
interchange the corresponding contours.
Hence, $I_2$ is also equal
to one half of the sum of the original expression and the
one where $u$ and $t$ are exchanged, that is,
\begin{multline*}
I_2=
\frac {1} {2\,(2\pi \mathbf{i})^2}\int _{\mathcal C_1} ^{}\int _{\mathcal C_2}
\frac {1} {(u-t)}
\frac {du} {u^{n+1}(1-u)^{n+1}}
\frac {dt} {t^{m}(1-t)^{m}}\\
-
\frac {1} {2\,(2\pi \mathbf{i})^2}\int _{\mathcal C_2} ^{}\int _{\mathcal C_1}
\frac {1} {(u-t)}
\frac {dt} {t^{n+1}(1-t)^{n+1}}
\frac {du} {u^{m}(1-u)^{m}}.
\end{multline*}
We would like to put both expressions under one integral.
In order to do so, we must blow up the contour $C_2$ in the
second integral (the contour for $t$) so that it passes
across $C_1$. When doing so, the term $u-t$ in the denominator
will vanish, and so we shall collect a residue at $t=u$.
This yields
{\allowdisplaybreaks
\begin{align*} \label{}
I_2&=
\frac {1} {2\,(2\pi \mathbf{i})^2}\int _{\mathcal C_1} ^{}\int _{\mathcal C_2}
\frac {du\,dt} {(u-t)\,\big(u(1-u)\,t(1-t)\big)^{n+1}}
\left(
\big(t(1-t)\big)^{n-m+1}
-
\big(u(1-u)\big)^{n-m+1}
\right)\\
&\kern1cm
+
\frac {1} {2\,(2\pi \mathbf{i})}\int _{\mathcal C_1}
\Res_{t=u}\frac {1} {(u-t)}
\frac {dt} {t^{n+1}(1-t)^{n+1}}
\frac {du} {u^{m}(1-u)^{m}}\\
&=
\frac {1} {2\,(2\pi \mathbf{i})^2}\int _{\mathcal C_1} ^{}\int _{\mathcal C_2}
\frac {du\,dt\,(u+t-1)} {\big(u(1-u)\,t(1-t)\big)^{n+1}}
\sum_{\ell=0} ^{n-m}
\big(t(1-t)\big)^{\ell}
\big(u(1-u)\big)^{n-m-\ell}\\
&\kern1cm
-\frac {1} {2\,(2\pi \mathbf{i})}\int _{\mathcal C_1}
\frac {du} {u^{n+m+1}(1-u)^{n+m+1}}\\
&=
\sum_{\ell=0} ^{n-m}
\frac {1} {2\,(2\pi \mathbf{i})^2}\int _{\mathcal C_1} ^{}\int _{\mathcal C_2}
\frac {du\,dt} {u^{m+\ell}(1-u)^{m+\ell+1}
\big(t(1-t)\big)^{n-\ell+1}}\\
&\kern1cm
-
\sum_{\ell=0} ^{n-m}
\frac {1} {2\,(2\pi \mathbf{i})^2}\int _{\mathcal C_1} ^{}\int _{\mathcal C_2}
\frac {du\,dt} {\big(u(1-u)\big)^{m+\ell+1}t^{n-\ell+1}(1-t)^{n-\ell}}
-\frac {1} {2}\binom {2n+2m}{n+m}\\
&=
\frac {1} {2}
\sum_{\ell=0} ^{n-m}
\binom {2n-2\ell} {n-\ell}
\binom {2m+2\ell-1} {m+\ell}
-
\frac {1} {2}
\sum_{\ell=0} ^{n-m}
\binom {2n-2\ell-1} {n-\ell-1}
\binom {2m+2\ell} {m+\ell}\\
&\kern1cm
-\frac {1} {2}\binom {2n+2m}{n+m}
=-\frac {1} {2}\binom {2n+2m}{n+m},
\end{align*}}%
the last equality following from $\binom {2k}k=2\binom {2k-1}k$.
\end{proof}

\begin{lemma} \label{lem:13}
For all non-negative integers $n$ and $m$ with $n\ge m$, we have
\begin{multline} \label{eq:neu4} 
\frac {1} {(2\pi \mathbf{i})^2}\int _{\mathcal C_1}\int _{\mathcal C_2}
^{}
\frac {1} {(u-t)(1-2t)}
\frac {du} {u^{n+1}(1-u)^{n+1}}
\frac {dt} {t^{m}(1-t)^{m}}\\
=-
\frac {1} {4}
\sum_{\ell=0} ^{n-m}
\binom {2n-2\ell}{n-\ell}
\binom {2m+2\ell} {m+\ell}
-
3\cdot2^{2n+2m-2},
\end{multline}
where $\mathcal C_1$ and $\mathcal C_2$ are contours close to $0$,
which encircle $0$ once in the positive direction, and $C_2$ is entirely
in the interior of $C_1$. The sum on the right-hand side must be
interpreted according to
\begin{equation} \label{eq:SUM} 
\sum _{k=M} ^{N-1}\text {\rm Expr}(k)=\begin{cases} 
\hphantom{-}
\sum _{k=M} ^{N-1} \text {\rm Expr}(k),&N>M,\\
\hphantom{-}0,&N=M,\\
-\sum _{k=N} ^{M-1}\text {\rm Expr}(k),&N<M.\end{cases}
\end{equation}
\end{lemma}

\begin{proof}
Again, here we treat the case where $n\ge m$. The other case can be disposed
of completely analogously.

Let $I_3$ denote the expression on the left-hand side of
\eqref{eq:neu4}. We apply the same trick as in the proof of
Lemma~\ref{lem:12} and observe that $I_3$ is equal
to one half of the sum of the original expression and the
one where $u$ and $t$ are exchanged, plus the residue of the latter 
at $t=u$. To be precise,
{\allowdisplaybreaks
\begin{align*} \label{}
I_3&=
\frac {1} {2\,(2\pi \mathbf{i})^2}\int _{\mathcal C_1} ^{}\int _{\mathcal C_2}
\frac {du\,dt}
      {(u-t)\,(1-2u)\,(1-2t)\,\big(u(1-u)\,t(1-t)\big)^{n+1}}\\
&\kern4cm
\cdot
\left(
(1-2u)\,\big(t(1-t)\big)^{n-m+1}
-
(1-2t)\,\big(u(1-u)\big)^{n-m+1}
\right)\\
&\kern1cm
+
\frac {1} {2\,(2\pi \mathbf{i})^2}\int _{\mathcal C_1}
^{}
\Res_{t=u}
\frac {1} {(u-t)(1-2u)}
\frac {1} {t^{n+1}(1-t)^{n+1}}
\frac {du} {u^{m}(1-u)^{m}}\\
&=
\frac {1} {2\,(2\pi \mathbf{i})^2}\int _{\mathcal C_1} ^{}\int _{\mathcal C_2}
\frac {du\,dt}
      {(u-t)\,(1-2t)\,\big(u(1-u)\,t(1-t)\big)^{n+1}}\\
&\kern4cm
\cdot
\left(
\big(t(1-t)\big)^{n-m+1}
-
\big(u(1-u)\big)^{n-m+1}
\right)\\
&\kern1cm
-
\frac {1} {(2\pi \mathbf{i})^2}\int _{\mathcal C_1} ^{}\int _{\mathcal C_2}
\frac {du\,dt}
      {(1-2u)\,(1-2t)\,\big(u(1-u)\big)^{m}\big(t(1-t)\big)^{n+1}}
\\
&\kern1cm
-
\frac {1} {2\,(2\pi \mathbf{i})^2}\int _{\mathcal C_1}
^{}
\frac {1} {(1-2u)}
\frac {du} {u^{n+m+1}(1-u)^{n+m+1}}\\
&=
\frac {1} {2\,(2\pi \mathbf{i})^2}\int _{\mathcal C_1} ^{}\int _{\mathcal C_2}
\frac {du\,dt\,(u+t-1)} {(1-2t)\,\big(u(1-u)\,t(1-t)\big)^{n+1}}
\sum_{\ell=0} ^{n-m}
\big(t(1-t)\big)^{\ell}
\big(u(1-u)\big)^{n-m-\ell}\\
&\kern4cm
-2^{2m-2+2n}-2^{2n+2m-1}\\
&=
\sum_{\ell=0} ^{n-m}
\frac {1} {2\,(2\pi \mathbf{i})^2}\int _{\mathcal C_1'} ^{}\int _{\mathcal C_2'}
\frac {du\,dt} {(1-2t)\,u^{m+\ell}(1-u)^{m+\ell+1}
\big(t(1-t)\big)^{n-\ell+1}}\\
&\kern1cm
-
\sum_{\ell=0} ^{n-m}
\frac {1} {2\,(2\pi \mathbf{i})^2}\int _{\mathcal C_1'} ^{}\int _{\mathcal C_2'}
\frac {du\,dt} {(1-2t)\,\big(u(1-u)\big)^{m+\ell+1}
t^{n-\ell+1}(1-t)^{n-\ell}}
-3\cdot2^{2m+2n-2}\\
&=
\frac {1} {2}
\sum_{\ell=0} ^{n-m}
\binom {2m+2\ell-1} {m+\ell}
2^{2n-2\ell}\\
&\kern1cm
-
\frac {1} {2}
\sum_{\ell=0} ^{n-m}
\binom {2m+2\ell} {m+\ell}
\left(2^{2n-2\ell-1}+\frac {1} {2}\binom {2n-2\ell}{n-\ell}\right)
-
3\cdot2^{2n+2m-2}\\
&=
-
\frac {1} {4}
\sum_{\ell=0} ^{n-m}
\binom {2n-2\ell}{n-\ell}
\binom {2m+2\ell} {m+\ell}
-
3\cdot2^{2n+2m-2},
\end{align*}}%
which is again seen by observing $\binom {2k}k=2\binom {2k-1}k$.
\end{proof}

We are now in the position to prove Lemma~\ref{lem:10} from
Section~\ref{sec:fund}.

\begin{proof}[Proof of Lemma~\ref{lem:10}]
Using complex contour integrals, we may write
\begin{align*} 
\sum_{0\le i\le j}
\binom {2n}{n+i}
\binom {2m}{m+j}
&=
\sum_{0\le i\le j}
\binom {2n}{n-i}
\binom {2m}{m-j}\\
&=\sum_{0\le i\le j}
\frac {1} {(2\pi \mathbf{i})^2}\int _{\mathcal C_1} ^{}\int _{\mathcal C_2}
^{}
\frac {(1+x)^{2n}} {x^{n-i+1}}
\frac {(1+y)^{2m}} {y^{m-j+1}}\,dx\,dy\\
&=\frac {1} {(2\pi \mathbf{i})^2}\int _{\mathcal C_1} ^{}\int _{\mathcal C_2}
^{}
\frac {(1+x)^{2n}} {x^{n+1}}
\frac {(1+y)^{2m}} {y^{m+1}}
\frac {dx\,dy} {(1-xy)(1-y)},
\end{align*}
where $\mathcal C_1$ and $\mathcal C_2$ are contours close to $0$,
which encircle $0$ once in the positive direction.

Now we do the substitutions $x=u/(1-u)$ and $y=t/(1-t)$,
implying $dx=du/(1-u)^2$ and $dy=dt/(1-t)^2$.
This leads to
\begin{align} 
\notag
\sum_{0\le i\le j}
\binom {2n}{n+i}
&\binom {2m}{m+j}\\
\notag
&=\frac {1} {(2\pi \mathbf{i})^2}\int _{\mathcal C_1'} ^{}\int _{\mathcal C_2'}
^{}
\frac {du} {u^{n+1}(1-u)^{n+1}}
\frac {dt} {t^{m+1}(1-t)^{m+1}}
\frac {(1-u)(1-t)^2} {(1-u-t)(1-2t)}\\
\notag
&=
\frac {1} {2\,(2\pi \mathbf{i})^2}\int _{\mathcal C_1'} ^{}\int _{\mathcal C_2'}
^{}
\frac {du} {u^{n+1}(1-u)^{n+1}}
\frac {dt} {t^{m+1}(1-t)^{m+1}}\\
\notag
&\kern1cm
-\frac {1} {(2\pi \mathbf{i})^2}\int _{\mathcal C_1'} ^{}\int _{\mathcal C_2'}
^{}
\frac {du} {u^{n+1}(1-u)^{n+1}}
\frac {dt} {t^{m}(1-t)^{m}}
\frac {1} {(1-2t)}\\
\notag
&\kern1cm
+
\frac {1} {2\,(2\pi \mathbf{i})^2}\int _{\mathcal C_1'} ^{}\int _{\mathcal C_2'}
^{}
\frac {du} {u^{n+1}(1-u)^{n+1}}
\frac {dt} {t^{m+1}(1-t)^{m+1}}
\frac {1} {(1-2t)}\\
\notag
&\kern1cm
+
\frac {1} {2\,(2\pi \mathbf{i})^2}\int _{\mathcal C_1'} ^{}\int _{\mathcal C_2'}
^{}
\frac {du} {u^{n+1}(1-u)^{n+1}}
\frac {dt} {t^{m}(1-t)^{m}}
\frac {1} {(1-u-t)}\\
&\kern1cm
+
\frac {1} {2\,(2\pi \mathbf{i})^2}\int _{\mathcal C_1'} ^{}\int _{\mathcal C_2'}
^{}
\frac {du} {u^{n+1}(1-u)^{n+1}}
\frac {dt} {t^{m}(1-t)^{m}}
\frac {1} {(1-u-t)(1-2t)}.
\label{eq:int0}
\end{align}
We now discuss the evaluation of the five integrals on the right-hand
side one by one. First of all, we have
\begin{align} 
\notag
\frac {1} {2\,(2\pi \mathbf{i})^2}\int _{\mathcal C_1'} ^{}\int _{\mathcal C_2'}
^{}
\frac {du} {u^{n+1}(1-u)^{n+1}}
\frac {dt} {t^{m+1}(1-t)^{m+1}}
&=\frac {1} {2}
\coef{u^n}(1-u)^{-n-1}\,\coef{t^m}(1-t)^{-m-1}\\
&=
\frac {1} {2}\binom {2n}n\binom {2m}m.
\label{eq:int1} 
\end{align}
Next, by Lemma~\ref{lem:11}, we have
\begin{equation} \label{eq:int2}
\frac {1} {(2\pi \mathbf{i})^2}\int _{\mathcal C_1'} ^{}\int _{\mathcal C_2'}
^{}
\frac {du} {u^{n+1}(1-u)^{n+1}}
\frac {dt} {t^{m}(1-t)^{m}}
\frac {1} {(1-2t)}
=2^{2m-2}\binom {2n}n
\end{equation}
and
\begin{equation} \label{eq:int3} 
\frac {1} {2\,(2\pi \mathbf{i})^2}\int _{\mathcal C_1'} ^{}\int _{\mathcal C_2'}
^{}
\frac {du} {u^{n+1}(1-u)^{n+1}}
\frac {dt} {t^{m+1}(1-t)^{m+1}}
\frac {1} {(1-2t)}
=
2^{2m-1}\binom {2n}n.
\end{equation}

In order to evaluate
\begin{equation*} 
I_4:=\frac {1} {2\,(2\pi \mathbf{i})^2}\int _{\mathcal C_1'} ^{}\int _{\mathcal C_2'}
^{}
\frac {du} {u^{n+1}(1-u)^{n+1}}
\frac {dt} {t^{m}(1-t)^{m}}
\frac {1} {(1-u-t)} ,
\end{equation*}
we blow up the contour $\mathcal C_1'$ (the contour for $u$) so that
it is sent to infinity. While doing this, we pass over
the poles $u=1-t$ and $u=1$ of the integrand. This must be
compensated by taking the residues at these points into account.
Since the integrand is of the order $O(u^{-2})$ as $\vert u\vert\to\infty$, 
the integral along the contour
near infinity vanishes. Thus, we obtain
\begin{align} 
\notag
I_4&=
-\frac {1} {2\,(2\pi \mathbf{i})}\int _{\mathcal C_2'} ^{}
\Res_{u=1-t}
\frac {1} {u^{n+1}(1-u)^{n+1}}
\frac {dt} {t^{m}(1-t)^{m}}
\frac {1} {(1-u-t)} \\
\notag
&\kern1cm
-\frac {1} {2\,(2\pi \mathbf{i})}\int _{\mathcal C_2'} ^{}
\Res_{u=1}
\frac {1} {u^{n+1}(1-u)^{n+1}}
\frac {dt} {t^{m}(1-t)^{m}}
\frac {1} {(1-u-t)} \\
\notag
&=
\frac {1} {2\,(2\pi \mathbf{i})}\int _{\mathcal C_2'} ^{}
\frac {dt} {t^{n+m+1}(1-t)^{n+m+1}}\\
\notag
&\kern1cm
-\frac {1} {2\,(2\pi \mathbf{i})^2}\int _{\mathcal C_1'} ^{}\int _{\mathcal C_2'}
\frac {du} {(1+u)^{n+1}(1-(1+u))^{n+1}}
\frac {dt} {t^{m}(1-t)^{m}}
\frac {1} {(1-(1+u)-t)} \\
&=\frac {1} {2}\binom {2n+2m} {n+m}-\frac {1} {4}\binom {2n+2m} {n+m}
=
\frac {1} {4}\binom {2n+2m} {n+m},
\label{eq:int4}
\end{align}
which is seen by performing the substitution $u\to -u$ in
the second expression in the next-to-last line and applying
Lemma~\ref{lem:12}.

Finally, in order to evaluate
\begin{equation} \label{eq:I_5}
I_5:=\frac {1} {2\,(2\pi \mathbf{i})^2}\int _{\mathcal C_1'} ^{}\int _{\mathcal C_2'}
^{}
\frac {du} {u^{n+1}(1-u)^{n+1}}
\frac {dt} {t^{m}(1-t)^{m}}
\frac {1} {(1-u-t)(1-2t)}
\end{equation}
we again blow up the contour $\mathcal C_1$ so that
it is sent to infinity. While doing this, we pass over
the poles $u=1-t$ and $u=1$ of the integrand. This must be
compensated by taking the residues at these points into account.
Since the integrand is of the order $O(u^{-2})$ as $\vert u\vert\to\infty$, 
the integral along the contour
near infinity vanishes. Thus, we obtain
\begin{align} 
\notag
I_5&=
-\frac {1} {2\,(2\pi \mathbf{i})}\int _{\mathcal C_2'}
^{}
\Res_{u=1-t}
\frac {1} {u^{n+1}(1-u)^{n+1}}
\frac {dt} {t^{m}(1-t)^{m}}
\frac {1} {(1-u-t)(1-2t)}\\
\notag
&\kern1cm
-\frac {1} {2\,(2\pi \mathbf{i})}\int _{\mathcal C_2'}
^{}
\Res_{u=1}
\frac {1} {u^{n+1}(1-u)^{n+1}}
\frac {dt} {t^{m}(1-t)^{m}}
\frac {1} {(1-u-t)(1-2t)}\\
\notag
&=
\frac {1} {2\,(2\pi \mathbf{i})}\int _{\mathcal C_2'}
^{}
\frac {dt} {t^{n+m+1}(1-t)^{n+m+1}}
\frac {1} {(1-2t)}\\
\notag
&\kern1cm
-\frac {1} {2\,(2\pi \mathbf{i})^2}\int _{\mathcal C_1'}\int _{\mathcal C_2'}
^{}
\frac {du} {(1+u)^{n+1}(1-(1+u))^{n+1}}
\frac {dt} {t^{m}(1-t)^{m}}
\frac {1} {(1-(1+u)-t)(1-2t)}\\
&=
2^{2n+2m-1}
-
\frac {1} {8}
\sum_{\ell=0} ^{n-m}
\binom {2n-2\ell}{n-\ell}
\binom {2m+2\ell} {m+\ell}
-
3\cdot2^{2n+2m-3},
\label{eq:int5}
\end{align}
which is seen by applying Lemma~\ref{lem:11} to the first expression
in the next-to-last line, performing the substitution $u\to-u$ in the second
expression, and applying Lemma~\ref{lem:13}.
By combining \eqref{eq:int0}--\eqref{eq:int5} and simplifying,
we obtain the right-hand side of \eqref{eq:neu1}.
\end{proof}

\section{Main results}
\label{sec:main}

This section contains our main results concerning double sums
of the form 
$$
\sum_{0\le i\le j}
i^sj^t
\binom {2n}{n+i}
\binom {2m}{m+j}.
$$
If both $s$ and $t$ are even, then we are only able to provide a
result in the special case where $m=n$. (It would also be possible to
provide a similar result for the case where the difference $n-m$ is
some fixed integer.) The reason is that the
identity in Lemma~\ref{lem:10}, 
on which an evaluation of the above sum will have to be based, contains the sum
over~$\ell$ that cannot be simplified if $n$ and $m$ are generic.
Proposition~\ref{prop:0} restricts attention to this special case.
On the other hand, if $s$ and $t$ are not both even, then it is possible
to provide a general result for the above double sum without any
restriction on $n$ and~$m$. The evaluations are then based on
Lemmas~\ref{lem:20}--\ref{lem:22}, and the corresponding results are
presented in Proposition~\ref{Prop:0}. It should be noted that, for
the three cases of parity of $s$ and $t$ that are treated in both
propositions, it is not true that Proposition~\ref{prop:0} is a
direct consequence of Proposition~\ref{Prop:0} as the assertions
in Proposition~\ref{prop:0} are more refined.

\begin{proposition} \label{prop:0}
For all non-negative integers $s,t,k$ and $n$, we have
\begin{multline} \label{eq:aux} 
\sum_{0\le i\le j\le n}
i^sj^t
\binom {2n}{n+i}
\binom {2n}{n+j}
=\frac {P^{(1)}_{s,t}(n)}{(4n-1)(4n-3)\cdots(4n-2S-2T+1)}{\binom {4n}{2n}}
\\
+
\frac {P^{(2)}_{s,t}(n)}{(2n-1)(2n-3)\cdots
(2n-2\fl{(S+T)/2}+1)}{\binom {2n}n}^2
+P^{(3)}_{s,t}(n)\cdot 4^n\binom {2n}n
+P^{(4)}_{s,t}(n)\cdot 16^n,
\end{multline}
where the $P^{(i)}_{s,t}(n)$, $i=1,2,3,4$, are polynomials in $n$,
$S=\fl{s/2}$ and $T=\fl{t/2}$.
More specifically,

\begin{enumerate} 
\item if $s$ and $t$ are even, then, as polynomials in $n$,
$P^{(1)}_{s,t}(n)$ is of degree at most $3S+3T$,
$P^{(2)}_{s,t}(n)$ is of degree at most 
$2S+2T+\fl{(S+T)/2}$,
$P^{(3)}_{s,t}(n)$ is identically zero if $s\ne0$,
$P^{(3)}_{0,t}(n)$ is of degree at most $2T$,
and $P^{(4)}_{s,t}(n)$ is of degree at most $2S+2T$;
\item if $s$ is odd and $t$ is even, then, as polynomials in $n$,
$P^{(1)}_{s,t}(n)$ is of degree at most $3S+3T+1$,
$P^{(2)}_{s,t}(n)$ is of degree at most 
$2S+2T+1+\fl{(S+T)/2}$,
$P^{(3)}_{s,t}(n)$ is of degree at most $2S+2T+1$,
and $P^{(4)}_{s,t}(n)$ is identically zero;
\item if $s$ is even and $t$ is odd, then, as polynomials in $n$,
$P^{(1)}_{s,t}(n)$ is of degree at most $3S+3T+1$,
$P^{(2)}_{s,t}(n)$ is of degree at most 
$2S+2T+1+\fl{(S+T)/2}$,
and $P^{(3)}_{s,t}(n)$ 
and $P^{(4)}_{s,t}(n)$ are identically zero;
\item if $s$ and $t$ are odd, then, as polynomials in $n$,
$P^{(1)}_{s,t}(n)$ is of degree at most $3S+3T+2$,
$P^{(2)}_{s,t}(n)$ is of degree at most 
$2S+2T+2+\fl{(S+T)/2}$,
and $P^{(3)}_{s,t}(n)$ 
and $P^{(4)}_{s,t}(n)$ are identically zero.
\end{enumerate}
\end{proposition}

\begin{remark} \label{rem:expl}
As the proof below shows, explicit formulae for the polynomials
$P^{(i)}_{s,t}(n)$, $i=1,2,3,4$, can be given that involve the coefficients
$c_{a,S}(n)$ and $c_{b,T}(n)$ in \eqref{eq:deci} and \eqref{eq:decj}, 
for which an explicit formula exists
as well, see Lemma~\ref{lem:6}. Admittedly, these explicit formulae
are somewhat cumbersome, and therefore we refrain from presenting them
in full here.
\end{remark}

\begin{proof}[Proof of Proposition~\ref{prop:0}]
We start with the case in which both $s$ and $t$ are even.
With the notation of the proposition, we have $s=2S$ and $t=2T$.
We write
\begin{equation} \label{eq:deci} 
i^{2S}
=
\sum_{a=0} ^{S}c_{a,S}(n)\,
\left(n^2-i^2\right)\left((n-1)^2-i^2\right)\cdots\left((n-a+1)^2-i^2\right),
\end{equation}
where $c_{a,S}(n)$ is a polynomial in $n$ of degree $2S-2a$,
$a=0,1,\dots,S$,
and
\begin{equation} \label{eq:decj} 
j^{2T}
=
\sum_{b=0} ^{T}c_{b,T}(n)
\left(n^2-j^2\right)\left((n-1)^2-j^2\right)\cdots\left((n-b+1)^2-j^2\right),
\end{equation}
where $c_{b,T}(n)$ is a polynomial in $n$ of degree $2T-2b$, $b=0,1,\dots,T$.
It should be noted that $c_{S,S}(n)=(-1)^S$ and $c_{T,T}(m)=(-1)^T$.
For an explicit formula for the coefficients $c_{a,S}(n)$ see 
Lemma~\ref{lem:6}.

If we use the expansions \eqref{eq:deci} and \eqref{eq:decj} 
on the left-hand side of \eqref{eq:aux},
then we obtain the expression
\begin{align*} 
\sum_{a=0} ^{S}\sum_{b=0} ^{T}&c_{a,S}(n)\,c_{b,T}(n)
\Bigg(
(2n-2a+1)_{2a}\,(2n-2b+1)_{2b}
\sum_{0\le i\le j}
\binom {2n-2a}{n+i-a}
\binom {2n-2b}{n+j-b}
\Bigg)\\
&=
\sum_{a=0} ^{S}\sum_{b=0} ^{T}c_{a,S}(n)\,c_{b,T}(n)\Bigg(
(2n-2a+1)_{2a}\,(2n-2b+1)_{2b}\\
&\kern2cm
\cdot
\left(
2^{4n-2a-2b-3}
+\frac {1} {4}\binom {4n-2a-2b} {2n-a-b}
+\frac {1} {2}\binom {2n-2a}{n-a}\binom {2n-2b}{n-b}\right.\\
&\kern2.5cm
\left.
+2^{2n-2b-2}\binom {2n-2a}{n-a}
-
\frac {1} {8}
\sum_{\ell=0} ^{b-a}
\binom {2n-2a-2\ell}{n-a-\ell}
\binom {2n-2b+2\ell} {n-b+\ell}\right)
\Bigg),
\end{align*}
due to Lemma~\ref{lem:10} with $n$ replaced by $n-a$ and $m=n-b$.
This expression can be further simplified by noting that
\begin{equation} \label{eq:0^{2S}} 
\sum_{a=0} ^S c_{a,S}(n)
\,(2n-2a+1)_{2a}\binom {2n-2a}{n-a}=
0^{2S}\binom {2n}n,
\end{equation}
which is equivalent to the expansion \eqref{eq:deci} for $i=0$.
Thus, we obtain
\begin{multline*}
\frac {1} {2}0^{2S+2T}{\binom {2n}{n}}^2
+0^{2S}\binom {2n}{n}\sum_{b=0} ^T c_{b,T}(n)\,
2^{2n-2b-2}\,(2n-2b+1)_{2b}\\
+\sum_{a=0} ^{S}\sum_{b=0} ^{T}c_{a,S}(n)\,c_{b,T}(n)\Bigg(
(2n-2a+1)_{2a}\,(2n-2b+1)_{2b}
\kern5cm\\
\cdot
\left(
2^{4n-2a-2b-3}
+\frac {1} {4}\binom {4n-2a-2b} {2n-a-b}
-
\frac {1} {8}
\sum_{\ell=0} ^{b-a}
\binom {2n-2a-2\ell}{n-a-\ell}
\binom {2n-2b+2\ell} {n-b+\ell}\right)
\Bigg).
\end{multline*}

Taking into account the properties of $c_{a,S}(n)$ and $c_{b,T}(n)$,
from this expression it is clear that $P^{(4)}_{s,t}(n)$,
the coefficient of $2^{4n}=16^n$, has degree at most $2S+2T$ as a
polynomial in~$n$.
It is furthermore obvious that, due to the term $0^{2S}=0^s$,
the polynomial $P^{(3)}_{s,t}(n)$,
the coefficient of $2^{2n}\binom {2n}n=4^n\binom {2n}n$, 
vanishes for $s\ne0$, while its degree is at most $2T$ if $s=0$.

In order to verify the claim about $P^{(1)}_{s,t}(n)$, the coefficient
of $\binom {4n}{2n}$, we write
\begin{multline*}
c_{a,S}(n)\,c_{b,T}(n)\,(2n-2a+1)_{2a}\,(2n-2b+1)_{2b}
\binom {4n-2a-2b} {2n-a-b}\\
=
c_{a,S}(n)\,c_{b,T}(n)\,
\frac {(2n-2a+1)_{2a}\,(2n-2b+1)_{2b}\,(2n-a-b+1)_{a+b}^2} 
{(4n-2a-2b+1)_{2a+2b}}
\binom {4n}{2n}.
\end{multline*}
It is easy to see that $(2n-a-b+1)_{a+b}$ divides numerator and
denominator. After this division, the denominator becomes
$$
2^{a+b}(4n-1)(4n-3)\cdots(4n-2a-2b+1),
$$
that is, part of the denominator below $P^{(1)}(n)$ in \eqref{eq:aux}.
The terms which are missing are
$$
(4n-2a-2b-1)(4n-2a-2b-3)\cdots (4n-2S-2T+1).
$$
Thus, if we put everything on the denominator
$$
(4n-1)(4n-3)\cdots(4n-2S-2T+1),
$$
then we see that the numerator of the coefficient of $\binom
{4n}{2n}$ has degree at most
$$
(2S-2a)+(2T-2b)+2a+2b+2(a+b)+(S+T-a-b)-(a+b)
=3S+3T,
$$ 
as desired.

Finally, 
we turn our attention to $P^{(2)}_{s,t}(n)$, the coefficient of ${\binom
  {2n}n}^2$. We have
\begin{subequations}
\begin{align} \notag
&c_{a,S}(n)\,c_{b,T}(n)\,(2n-2a+1)_{2a}\,(2n-2b+1)_{2b}
\binom {2n-2a-2\ell}{n-a-\ell}
\binom {2n-2b+2\ell} {n-b+\ell}\\
\label{eq:2nn^2a}
&\qquad 
=
c_{a,S}(n)\,c_{b,T}(n)\,
\frac {(n-a-\ell+1)_{a+\ell}^2\,(n-b+\ell+1)_{b-\ell}^2\,(2n-2b+1)_{2\ell}} 
{(2n-2a-2\ell+1)_{2\ell}}
{\binom {2n}n}^2\\
&\qquad 
=
c_{a,S}(n)\,c_{b,T}(n)\,
\frac {(n-a-\ell+1)_{a+\ell}^2\,(n-b+\ell+1)_{b-\ell}^2\,
(2n-2b+1)_{2b-2a-2\ell}} 
{(2n-2b+2\ell+1)_{2b-2a-2\ell}}
{\binom {2n}n}^2.
\label{eq:2nn^2b}
\end{align}
\end{subequations}
Let us assume $a\le b$, in which case we need to
consider non-negative indices~$\ell$. (If $a>b$, then, according to
the convention \eqref{eq:SUM}, we have to consider negative~$\ell$.
Using the definition \eqref{eq:Poch} of the Pochhammer symbol for
negative indices, the arguments would be completely analogous.)
We make the further assumption that $\ell\le\frac {1} {2}(b-a)$
and use expression \eqref{eq:2nn^2a}. (If $\ell>\frac {1} {2}(b-a)$,
then analogous arguments work starting from expression
\eqref{eq:2nn^2b}.)

It is easy to see that $(n-a-\ell+1)_\ell$ divides numerator and
denominator (as polynomials in~$n$) of the prefactor in \eqref{eq:2nn^2a}. 
Second, the (remaining) factor $2^{2\ell}(n-a-\ell+\frac {1} {2})_{\ell}$ 
in the denominator and
the factor $(2n-2b+1)_{2\ell}$ in the numerator do not have common
factors for
$\ell\le \frac {1} {2}(b-a)$. 
The denominator is
a factor of the denominator below $P^{(2)}_{s,t}(n)$ in \eqref{eq:aux}.
If in \eqref{eq:2nn^2a} we extend denominator and numerator by
the ``missing" factor 
$$
(n-\fl{(S+T)/2}+\tfrac {1} {2})_{\fl{(T+S)/2}-\fl{b+a}/2}\,
(n-a+\tfrac {1} {2})_{a},
$$
then, due to the properties of $c_{a,S}(n)$ and $c_{b,T}(n)$, the
numerator polynomial is of degree at most
\begin{align*}
(2S-2a)+(2T-2b)&+2(a+\ell)+2(b-\ell)+2\ell-\ell+\fl{(T+S)/2}-\fl{(b+a)/2}+a\\
&=
2S+2T+\ell+\fl{(T+S)/2}-\fl{(b+a)/2}+a\\
&\le
2S+2T+\fl{(b-a)/2}+\fl{(T+S)/2}-\fl{(b+a)/2}+a\\
&\le
2S+2T+\fl{(S+T)/2},
\end{align*}
as desired.

\medskip
For the other cases, namely $(s,t)$ being (odd,\,even), (even,\,odd),
respectively (odd,\,odd), we proceed in the same way.
That is, we apply the expansions \eqref{eq:deci} and
\eqref{eq:decj} on the left-hand side of \eqref{eq:aux}.
Then, however, instead of Lemma~\ref{lem:10}, we apply 
Lemma~\ref{lem:20}, Lemma~\ref{lem:21}, and Lemma~\ref{lem:22},
respectively. The remaining arguments are completely analogous to
those from the case of $(s,t)$ being (even,even) (and, in fact, much simpler
since the right-hand sides of the identities in Lemmas~\ref{lem:20}--\ref{lem:22}
are simpler than the one in Lemma~\ref{lem:10}).
\end{proof}

\begin{proposition} \label{Prop:0}
Let $s,t$ and $n,m$ be non-negative integers. 

If $s$ and $t$ are not both even or both odd, then
\begin{multline} \label{Eq:aux} 
\sum_{0\le i\le j}
i^sj^t
\binom {2n}{n+i}
\binom {2m}{m+j}\\
=\frac {Q^{(1)}_{s,t}(n,m)}{(2n+2m-1)(2n+2m-3)\cdots(2n+2m-2S-2T+1)}
{\binom {2n+2m}{n+m}}
\\
\kern1cm
+
\frac {Q^{(2)}_{s,t}(n,m)}{(n+m)(n+m-1)(n+m-2)\cdots
(n+m-S-T)}{\binom {2n}n}\binom {2m}m\\
+Q^{(3)}_{s,t}(n,m)\cdot 4^m\binom {2n}n,
\end{multline}
where the $Q^{(i)}_{s,t}(n,m)$, $i=1,2,3$, are polynomials in $n$
and~$m$, $S=\fl{s/2}$ and $T=\fl{t/2}$.
More specifically,

\begin{enumerate} 
\item if $s$ is odd and $t$ is even, then, as polynomials in $n$ and~$m$,
$Q^{(1)}_{s,t}(n,m)$ is of degree at most $3S+3T+1$,
$Q^{(2)}_{s,t}(n,m)$ is of degree at most 
$3S+3T+2$, and
$Q^{(3)}_{s,t}(n,m)$ is of degree at most $2S+2T+1$;
\item if $s$ is even and $t$ is odd, then, as polynomials in $n$ and~$m$,
$Q^{(1)}_{s,t}(n,m)$ is of degree at most $3S+3T+1$,
$Q^{(2)}_{s,t}(n,m)$ is of degree at most 
$3S+3T+2$,
and $Q^{(3)}_{s,t}(n,m)$ is identically zero.
\end{enumerate}

If $s$ and $t$ are odd, then
\begin{multline} \label{Eq:aux2} 
\sum_{0\le i\le j}
i^sj^t
\binom {2n}{n+i}
\binom {2m}{m+j}\\
=\frac {Q^{(1)}_{s,t}(n,m)}{(2n+2m-1)(2n+2m-3)\cdots(2n+2m-2S-2T-1)}
{\binom {2n+2m}{n+m}}
\\
\kern1cm
+
\frac {Q^{(2)}_{s,t}(n,m)}{(n+m)(n+m-1)(n+m-2)\cdots
(n+m-S-T)}{\binom {2n}n}\binom {2m}m,
\end{multline}
where $S=\fl{s/2}$ and $T=\fl{t/2}$, and,
as polynomials in $n$ and~$m$,
$Q^{(1)}_{s,t}(n,m)$ 
and $Q^{(2)}_{s,t}(n,m)$ are of degree at most 
$3S+3T+3$.
\end{proposition}

The proof of this proposition is completely analogous to the proof of
Proposition~\ref{prop:0} and is therefore left to the reader.
Also here (cf.\ Remark~\ref{rem:expl}), explicit formulae for the polynomials
$Q^{(i)}_{s,t}(n,m)$, $i=1,2,3$, can be given that involve coefficients
$c_{a,S}(n)$ and $c_{b,T}(m)$ for which an explicit formula exists
(see Lemma~\ref{lem:6}). 

\section{Some more auxiliary results}
\label{sec:aux2}

In this section we derive some single sum evaluations that we shall
need in the proofs in Section~\ref{sec:Gl}.

\begin{lemma} \label{lem:5}
For all non-negative integers $n$ and $k$, we have
\begin{equation} \label{eq:aux8} 
\sum_{j=1} ^n
j^{2k}\binom {2n}{n+j}
=-
\frac {0^{2k}} {2}\binom {2n} {n}
+
4^n\,\sum_{b=0} ^k c_{b,k}(n)\,(2n-2b+1)_{2b}\,
2^{-2b-1},
\end{equation}
and
\begin{equation} \label{eq:aux9} 
\sum_{j=1} ^n
j^{2k+1}\binom {2n}{n+j}
=
\frac {1} {2}\binom {2n}{n}
\sum_{b=0} ^k c_{b,k}(n)\,(n-b)_{b+1}\,(n-b+1)_b\,,
\end{equation}
where the coefficients $c_{b,k}(n)$ are defined in \eqref{eq:decj}
{\em(}with explicit formula provided in Lem\-ma~\ref{lem:6}{\em)}.
\end{lemma}

\begin{proof}
We use the expansion \eqref{eq:decj} with $T=k$ on the
left-hand side of \eqref{eq:aux8}. This gives
\begin{align*}
\sum_{j=1} ^n
j^{2k}\binom {2n}{n+j}
&=
\sum_{j=1} ^n
\sum_{b=0} ^k c_{b,k}(n)\,(2n-2b+1)_{2b}\binom {2n-2b} {n+j-b}\\
&=
\sum_{b=0} ^k c_{b,k}(n)\,(2n-2b+1)_{2b}
\left(
2^{2n-2b-1}-
\frac {1} {2}\binom {2n-2b} {n-b}
\right)\\
&=
-
\frac {0^{2k}} {2}\binom {2n} {n}
+
\sum_{b=0} ^k c_{b,k}(n)\,(2n-2b+1)_{2b}\,
2^{2n-2b-1},
\end{align*}
where we used \eqref{eq:decj} with $T=k$ and $j=0$ in the last line.
This is exactly the right-hand side of \eqref{eq:aux8}.

Now we do the same on the
left-hand side of \eqref{eq:aux9}. This leads to
\begin{align*}
\sum_{j=1} ^n
j^{2k+1}\binom {2n}{n+j}
&=
\sum_{j=1} ^n
j\cdot \sum_{b=0} ^k c_{b,k}(n)\,(2n-2b+1)_{2b}\binom {2n-2b} {n+j-b}\\
&=
\sum_{b=0} ^k c_{b,k}(n)\,(2n-2b)_{2b+1}
\sum_{j=1} ^n
\left(
\binom {2n-2b-1} {n+j-b-1}
-\frac {1} {2}\binom {2n-2b} {n+j-b}
\right)\\
&=
\sum_{b=0} ^k c_{b,k}(n)\,(2n-2b)_{2b+1}
\left(
2^{2n-2b-2}
-\frac {1} {2}2^{2n-2b-1}+\frac {1} {4}\binom {2n-2b}{n-b}
\right)\\
&=
\frac {1} {2}\sum_{b=0} ^k c_{b,k}(n)\,(n-b)_{b+1}\,(n-b+1)_b
\binom {2n}{n}.
\end{align*}
This is exactly the right-hand side of \eqref{eq:aux9}.
\end{proof}

\begin{lemma} \label{lem:4}
For all non-negative integers $n$ and $h,k$, we have
\begin{equation} \label{eq:aux7a} 
\sum_{j\ge1} 
j^{2k}
{\binom {2n}{n+j}}{\binom {2m}{m+j}}
=
-\frac {0^{2k}} {2}{\binom {2n}n}{\binom {2m}m}
+\frac {1} {2}
\sum_{b=0} ^k c_{b,k}(n)\,(2n-2b+1)_{2b}
\binom {2n+2m-2b}{n+m-b}
\end{equation}
and
\begin{multline} \label{eq:aux7b} 
\sum_{j\ge1}
j^{2h+2k+1}
{\binom {2n}{n+j}}{\binom {2m}{m+j}}\\
=
\sum_{a=0} ^h\sum_{b=0} ^k 
c_{a,h}(n)\,c_{b,k}(m)\,(n-a+1)_a^2\,(m-b+1)_b^2\,
\frac {(n-a)(m-b)} {2(n+m-a-b)}{\binom {2n} {n}}{\binom {2m} {m}},
\end{multline}
where the coefficients $c_{a,h}(n)$ and 
$c_{b,k}(m)$ are defined in \eqref{eq:deci}
{\em(}with explicit formula provided in Lemma~\ref{lem:6}{\em)}.
\end{lemma}

\begin{proof}
We start by using the expansion \eqref{eq:decj} with $T=k$ on the
left-hand side of \eqref{eq:aux7a}. This gives
\begin{equation}
\sum_{j\ge1} 
j^{2k}{\binom {2n}{n+j}}{\binom {2m}{m+j}}
=
\sum_{j\ge1} 
\sum_{b=0} ^k c_{b,k}(n)\,(2n-2b+1)_{2b}\binom {2n-2b} {n+j-b}
\binom {2m}{m+j}.
\label{eq:aux7c}
\end{equation}
We have
$$
\sum_{j\ge1}
\binom {2n-2b}{n+j-b}
\binom {2m}{m+j}
=
\sum_{j\le-1}
\binom {2n-2b}{n+j-b}
\binom {2m}{m+j}
$$
and hence
\begin{align*}
\sum_{j\ge1}
\binom {2n-2b}{n+j-b}
\binom {2m}{m+j}
&=-\frac {1} {2}
\binom {2n-2b}{n-b}
\binom {2m}{m}
+\frac {1} {2}
\sum_{j} 
\binom {2n-2b}{n+j-b}
\binom {2m}{m+j}\\
&=-\frac {1} {2}
\binom {2n-2b}{n-b}
\binom {2m}{m}
+\frac {1} {2}
\binom {2n+2m-2b}{n+m-b},
\end{align*}
due to the Chu--Vandermonde summation.
We substitute this back into \eqref{eq:aux7c} and obtain
\begin{align*}
\sum_{j\ge1}
j^{2k}{\binom {2n}{n+j}}{\binom {2m}{m+j}}
&=
\sum_{b=0} ^k c_{b,k}(n)\,(2n-2b+1)_{2b}\\
&\kern3cm
\cdot
\left(
-\frac {1} {2}
\binom {2n-2b}{n-b}
\binom {2m}{m}
+\frac {1} {2}
\binom {2n+2m-2b}{n+m-b}
\right)\\
&\kern-1cm
=
-\frac {0^{2k}} {2}{\binom {2n}n}{\binom {2m}m}
+\frac {1} {2}
\sum_{b=0} ^k c_{b,k}(n)\,(2n-2b+1)_{2b}
\binom {2n+2m-2b}{n+m-b},
\end{align*}
where we used \eqref{eq:decj} with $T=k$ and $j=0$ in the last line.

In order to establish \eqref{eq:aux7b}, 
we write $j^{2h+2k+1}=j\cdot j^{2h}\cdot j^{2k}$ and use \eqref{eq:decj}
with $T=h$ and with $T=k$. This leads to
\begin{multline} \label{eq:aux7d} 
\sum_{j=1} ^n
j^{2h+2k+1}\,
{\binom {2n}{n+j}}{\binom {2m}{m+j}}\\
=
\sum_{j\ge1}
j\cdot
\sum_{a=0} ^h\sum_{b=0} ^k c_{a,h}(n)\,c_{b,k}(m)\,
(2n-2a+1)_{2a}\,(2m-2b+1)_{2b}
\binom {2n-2a} {n+j-a}
\binom {2m-2b}{m+j-b}.
\end{multline}
Using the standard hypergeometric notation
$${}_p F_q\!\left[\begin{matrix} a_1,\dots,a_p\\ b_1,\dots,b_q\end{matrix}; 
z\right]=\sum _{m=0} ^{\infty}\frac {\po{a_1}{m}\cdots\po{a_p}{m}}
{m!\,\po{b_1}{m}\cdots\po{b_q}{m}} z^m\ ,
$$
where the Pochhammer symbol $(\alpha)_m$ is defined in
\eqref{eq:Poch}, we have
\begin{multline*}
\sum_{j\ge1}
j\,\binom {2n-2a} {n+j-a}
\binom {2m-2b}{m+j-b}\\
=\binom {2n-2a} {n-a+1}
\binom {2m-2b}{m-b+1}\,
{} _{3} F _{2} \!\left [ \begin{matrix} 
2,-n+a+1,-m+b+1\\
n-a+2,m-b+2
\end{matrix} ; {\displaystyle
       1}\right ].
\end{multline*}
This $_3F_2$-series can be evaluated by means of 
(the terminating
version of) Dixon's summation (see \cite[Appendix~(III.9)]{SlatAC})
$$
{} _{3} F _{2} \!\left [ \begin{matrix} 
{A,B,-N}\\{1+A-B,1+A+N}\end{matrix};1\right]=
\frac{(1+A)_N\,(1+\frac{A}{2}-B)_N}{(1+\frac{A}{2})_N\,(1+A-B)_N},
$$
where $N$ is a non-negative integer.
Indeed, if we choose $A=2$, $B=-n+a+1$, and $N=m-b-1$ in this summation
formula, then we obtain
$$
\sum_{j=1} ^n
j\,\binom {2n-2a} {n+j-a}
\binom {2m-2b}{m+j-b}
=\frac {(n-a+1)(m-b+1)} {2(n+m-a-b)}
\binom {2n-2a} {n-a+1}
\binom {2m-2b}{m-b+1}.
$$
If this is substituted back in \eqref{eq:aux7d}, then we obtain
the right-hand side of \eqref{eq:aux7b} after little manipulation.
\end{proof}

For the proof of our theorems it is not necessary to have an explicit formula
for the coefficients $c_{a,S}(n)$ in the expansion \eqref{eq:deci}
--- the coefficients 
that appeared in the proof of Proposition~\ref{prop:0}, and
in Lemmas~\ref{lem:5} and \ref{lem:4} --- at
our disposal. However, it is still of intrinsic interest to provide
such an explicit formula.

\begin{lemma} \label{lem:6}
The coefficient $c_{a,S}(n)$ in the expansion \eqref{eq:deci} is given by
\begin{equation} \label{eq:aux10} 
c_{a,S}(n)=\sum_{r=0}^a \frac {2(-1)^{a+r}(n-r)^{2S+1}} 
            {r!\,(a-r)!\,(2n-a-r)_{a+1}}.
\end{equation}
\end{lemma}

\begin{proof}
Substituting $n-b$ for $i$ in \eqref{eq:deci}, $b=0,1,\dots,S$, we
obtain the triangular system of linear equations for the $c_{a,S}$'s
\begin{equation} \label{eq:2Sc} 
(n-b)^{2S}
=
\sum_{a=0} ^{b}c_{a,S}(n)\,(2n-a-b+1)_a\,(b-a+1)_a,
\quad b=0,1,\dots,S.
\end{equation}
Now, using the inversion formula of Gould and Hsu \cite{GOHsAA}
(in the statement\break \cite[Eq.~(1.1)]{KratAN} of the formula 
put $n=b$, $k=a$,
$l=r$, $a_j=2n-j$, $b_j=-1$, in this order), we see that the
matrix
$$
\big((2n-a-b+1)_a\,(b-a+1)_a\big)_{b,a\ge0}
$$
is inverse to the matrix
$$
\left(\frac {(-1)^{a+r}(2n-2r)} 
            {r!\,(a-r)!\,(2n-a-r)_{a+1}}\right)_{a,r\ge0}.
$$
Hence, if the system \eqref{eq:2Sc} is inverted, that is, the
coefficients $c_{a,S}(n)$, $a=0,1,\dots,S$, 
are expressed in terms of the $(n-b)^{2S}$, $b=0,1,\dots,S$,
then the connection coefficients are given by the latter matrix.
This proves \eqref{eq:aux10}.
\end{proof}

\section{Summation formulae for binomial double sums involving
  absolute values}
\label{sec:Gl}

In this section we present the implications of
Propositions~\ref{prop:0} and \ref{Prop:0} on sums of the form~\eqref{eq:sumallg}
and \eqref{eq:sumallg2}
with $\beta=1$. As we point out in Remark~\ref{rem:1}(1) below,
it would also be possible to derive similar theorems for
arbitrary~$\beta$.
(An example of an evaluation with $\beta=3$ is given in \eqref{eq:beta=3}.)

We start with results for double sums of the form
\eqref{eq:sumallg2} with even~$k$ (and $\beta=1$). First, we also let $m=n$.
The corresponding evaluations are given in Theorem~\ref{thm:2}
below. In Theorem~\ref{MNthm:2} we address these same double
sums for generic $n$ and~$m$. Similarly to Proposition~\ref{Prop:0},
for that case we
have results only if $s$ and $t$ are not both even.

\begin{theorem} \label{thm:2}
Let $s,t,k$ and $n$ be  non-negative integers.

If $s$ and $t$ are even, then
\begin{multline} \label{eq:2k1} 
\sum_{-n\le i,j\le n}\left\vert i^sj^t(j^{2k}-i^{2k})\right\vert
\binom {2n}{n+i}
\binom {2n}{n+j}\\
=\frac {U^{(2)}_{s,t,k}(n)}{(2n-1)(2n-3)\cdots(2n-2\fl{(S+T+k)/2}+1)}
{\binom {2n}n}^2,
\end{multline}
where $U^{(2)}_{s,t,k}(n)$ is of degree at most 
$2S+2T+2k+\fl{(S+T+k)/2}$.

If $s$ and $t$ are both odd, then
\begin{multline} \label{eq:2k2} 
\sum_{-n\le i,j\le n}\left\vert i^sj^t(j^{2k}-i^{2k})\right\vert
\binom {2n}{n+i}
\binom {2n}{n+j}\\
=\frac {U^{(2)}_{s,t,k}(n)}{(2n-1)(2n-3)\cdots(2n-2\cl{(S+T+k)/2}+1)}
{\binom {2n}n}^2,
\end{multline}
where $U^{(2)}_{s,t,k}(n)$ is of degree at most 
$2S+2T+2k+\cl{(S+T+k)/2}$.

If $s$ and $t$ have different parity, then
\begin{multline} \label{eq:2k3} 
\sum_{-n\le i,j\le n}\left\vert i^sj^t(j^{2k}-i^{2k})\right\vert
\binom {2n}{n+i}
\binom {2n}{n+j}\\
=\frac {U^{(1)}_{s,t,k}(n)}{(4n-1)(4n-3)\cdots(4n-2S-2T-2k+1)}
{\binom {4n}{2n}}
+U^{(3)}_{s,t,k}(n)\cdot 4^n\binom {2n}n,
\end{multline}
where $U^{(1)}_{s,t,k}(n)$ and $U^{(3)}_{s,t,k}(n)$ are polynomials in $n$,
$S=\fl{s/2}$ and $T=\fl{t/2}$.

More specifically,

\begin{enumerate} 
\item if $s$ is odd and $t$ is even, then, as polynomials in $n$,
$U^{(1)}_{s,t,k}(n)$ is of degree at most 
$3S+3T+3k+1$, and
$U^{(3)}_{s,t,k}(n)$ is of degree at most $2S+2T+2k+1$;
\item if $s$ is even and $t$ is odd, then, as polynomials in $n$,
$U^{(1)}_{s,t,k}(n)$ is of degree at most 
$3S+3T+3k+1$, and
$U^{(3)}_{s,t,k}(n)$ is of degree at most $2S+2T+2k+1$.
\end{enumerate}
\end{theorem}

\begin{remarkno}
As the proof below shows, also here (cf.\ Remark~\ref{rem:expl}) 
explicit formulae for the polynomials
$U^{(i)}_{s,t,k}(n)$, $i=1,2,3$, can be given that involve coefficients
$c_{a,A}(n)$, for various specific choices of $A$.
As we pointed out at several places already, Lemma~\ref{lem:6} provides 
an explicit formula for these coefficients.
\end{remarkno}

\begin{proof}[Proof of Theorem~\ref{thm:2}]
The claim is trivially true for $k=0$. Therefore we may assume
from now on that $k>0$.

Using the operations 
$(i,j)\to(-i,j)$,
$(i,j)\to(i,-j)$,
and $(i,j)\to(j,i)$,
which do not change the summand, we see that
\begin{align} \notag
\sum_{-n\le i,j\le n}
&\left\vert i^sj^t(j^{2k}-i^{2k})\right\vert
\binom {2n}{n+i}
\binom {2n}{n+j}\\
\notag
&=4\sum_{0\le i\le j\le n}
\alpha(i=0)\,\alpha(j=0)\,
\left(i^sj^t+i^tj^s\right)\left( j^{2k}-i^{2k}\right)
\binom {2n}{n+i}
\binom {2n}{n+j}\\
\notag
&=4\sum_{0\le i\le j\le n}
\left(i^sj^t+i^tj^s\right)\left( j^{2k}-i^{2k}\right)
\binom {2n}{n+i}
\binom {2n}{n+j}\\
&\kern2cm
-2\binom {2n}{n}
\sum_{j=1}^n
\left(0^sj^t+0^tj^s\right)j^{2k}
\binom {2n}{n+j},
\label{eq:7} 
\end{align}
where $\alpha(\mathcal A)=\frac {1} {2}$ if $\mathcal A$ is true 
and $\alpha(\mathcal A)=1$ otherwise.
Now one splits the sums into several sums of the form 
$$
\sum_{0\le i\le j\le n}
i^Aj^B
\binom {2n}{n+i}
\binom {2n}{n+j},
\quad \text{respectively}\quad 
\sum_{j=1}^n
j^{B}
\binom {2n}{n+j}.
$$
To sums of the second form, we apply Lemma~\ref{lem:5}.
In order to evaluate the sums of the first form, we proceed as in
the proof of Proposition~\ref{prop:0}. 
That is, we apply the expansions \eqref{eq:deci} and
\eqref{eq:decj},
and subsequently we use Lemmas~\ref{lem:10}--\ref{lem:22} to
evaluate the sums over $i$ and $j$. Inspection of the result
makes all assertions of the theorem obvious, except for the
implicit claims in \eqref{eq:2k1} and \eqref{eq:2k2} that
the term $4^n\binom {2n}n$ does not appear.

In order to verify these claims, we have to figure out
what the coefficients of $4^n\binom {2n}n$ of the various sums in
\eqref{eq:7} are precisely. For the case of even $s$ and $t$,
from Lemma~\ref{lem:10} we obtain that the coefficient of
$4^n\binom {2n}n$ in the expression \eqref{eq:7} equals
{\allowdisplaybreaks
\begin{multline*} 
4\sum_{a=0} ^{S}
\sum_{b=0} ^{T+k}
c_{a,S}(n)\,c_{b,T+k}(n) \,(2n-2a+1)_{2a}\,(2n-2b+1)_{2b}\,
2^{-2b-2}\binom {2n-2a}{n-a}{\binom {2n}n}^{-1}\\
+
4\sum_{a=0} ^{T}
\sum_{b=0} ^{S+k}
c_{a,T}(n)\,c_{b,S+k}(n) \,(2n-2a+1)_{2a}\,(2n-2b+1)_{2b}\,
2^{-2b-2}\binom {2n-2a}{n-a}{\binom {2n}n}^{-1}\\
-4\sum_{a=0} ^{S+k}
\sum_{b=0} ^{T}
c_{a,S+k}(n)\,c_{b,T}(n) \,(2n-2a+1)_{2a}\,(2n-2b+1)_{2b}\,
2^{-2b-2}\binom {2n-2a}{n-a}{\binom {2n}n}^{-1}\\
-4\sum_{a=0} ^{T+k}
\sum_{b=0} ^{S}
c_{a,T+k}(n)\,c_{b,S}(n) \,(2n-2a+1)_{2a}\,(2n-2b+1)_{2b}\,
2^{-2b-2}\binom {2n-2a}{n-a}{\binom {2n}n}^{-1}\\
-2\cdot0^{2S}
\sum_{b=0} ^{T+k}c_{b,T+k}(n) \,(2n-2b+1)_{2b}\,
2^{-2b-1}\\
-2\cdot0^{2T}
\sum_{b=0} ^{S+k}c_{b,S+k}(n) \,(2n-2b+1)_{2b}\,
2^{-2b-1}.
\end{multline*}}%
We may use \eqref{eq:0^{2S}} to simplify the double sums.
In this manner, we arrive at the expression
\begin{multline*} 
0^{2S}
\sum_{b=0} ^{T+k}
c_{b,T+k}(n) \,(2n-2b+1)_{2b}\,
2^{-2b}
+
0^{2T}
\sum_{b=0} ^{S+k}
c_{b,S+k}(n) \,(2n-2b+1)_{2b}\,
2^{-2b}
\kern1cm\\
\qquad -
0^{2S+2k}
\sum_{b=0} ^{T}
c_{b,T}(n) \,(2n-2b+1)_{2b}\,
2^{-2b}
-
0^{2T+2k}
\sum_{b=0} ^{S}
c_{b,S}(n) \,(2n-2b+1)_{2b}\,
2^{-2b}\\
\qquad 
-0^{2S}
\sum_{b=0} ^{T+k}c_{b,T+k}(n) \,(2n-2b+1)_{2b}\,
2^{-2b}
-0^{2T}
\sum_{b=0} ^{S+k}c_{b,S+k}(n) \,(2n-2b+1)_{2b}\,
2^{-2b},
\end{multline*}
which visibly vanishes due to our assumption that $k>0$.

The proof for the analogous claim in the case of odd $s$ and $t$ 
proceeds along the
same lines. The only difference is that, instead of
Lemma~\ref{lem:10}, here we need Lemma~\ref{lem:22}, and instead
of \eqref{eq:aux8} we need \eqref{eq:aux9}.
\end{proof}

\begin{theorem} \label{MNthm:2}
Let $s,t,k$ and $n,m$ be non-negative integers.
If $s$ and $t$ are not both even, then
\begin{multline} \label{MNeq:2k1} 
\sum_{i,j}\left\vert i^sj^t(j^{2k}-i^{2k})\right\vert
\binom {2n}{n+i}
\binom {2m}{m+j}\\
=\frac {V^{(1)}_{s,t,k}(n,m)}{(2n+2m-1)(2n+2m-3)\cdots(2n+2m-2S-2T-2k+1)}
{\binom {2n+2m}{n+m}}
\\
\kern1cm
+
\frac {V^{(2)}_{s,t,k}(n,m)}{(n+m-1)(n+m-2)\cdots
(n+m-S-T-k)}{\binom {2n}n}\binom {2m}m\\
+V^{(3)}_{s,t,k}(n,m)\cdot 4^m\binom {2n}n
+V^{(4)}_{s,t,k}(n,m)\cdot 4^n\binom {2m}m,
\end{multline}
where the $V^{(i)}_{s,t,k}(n,m)$, $i=1,2,3,4$, are polynomials in $n$ and~$m$,
$S=\fl{s/2}$ and $T=\fl{t/2}$.

More specifically,

\begin{enumerate} 
\item if $s$ is odd and $t$ is even, then, as polynomials in $n$ and~$m$,
$V^{(1)}_{s,t,k}(n,m)$ is of degree at most 
$3S+3T+3k+1$,
$V^{(3)}_{s,t,k}(n,m)$ is of degree at most $2S+2T+2k+1$,
and $V^{(2)}_{s,t,k}(n,m)$ and $V^{(4)}_{s,t,k}(n,m)$ are identically zero,
\item if $s$ is even and $t$ is odd, then, as polynomials in $n$ and~$m$,
$V^{(1)}_{s,t,k}(n,m)$ is of degree at most 
$3S+3T+3k+1$,
$V^{(4)}_{s,t,k}(n,m)$ is of degree at most $2S+2T+2k+1$,
and $V^{(2)}_{s,t,k}(n,m)$ and $V^{(3)}_{s,t,k}(n,m)$ are identically zero,
\item if $s$ and $t$ are odd, then, as polynomials in $n$ and~$m$,
$V^{(2)}_{s,t,k}(n,m)$ is of degree at most 
$3S+3T+3k+2$,
and $V^{(1)}_{s,t,k}(n,m)$, $V^{(3)}_{s,t,k}(n,m)$, and
$V^{(4)}_{s,t,k}(n,m)$ are identically zero.
\end{enumerate}
\end{theorem}

\begin{remarkno}
Again (cf.\ Remark~\ref{rem:expl}), from the proof below it is obvious that 
explicit formulae for the polynomials
$V^{(i)}_{s,t,k}(n,m)$, $i=1,2,3,4$, are available in terms of coefficients
$c_{a,A}(n)$ and $c_{b,B}(m)$, for various specific choices of $A$ and~$B$,
with Lemma~\ref{lem:6} providing an explicit formula for these coefficients.
\end{remarkno}

\begin{proof}[Proof of Theorem~\ref{MNthm:2}]
Again, the claim is trivially true for $k=0$. Therefore we may assume
from now on that $k>0$.

We follow the same idea as in the proof of Theorem~\ref{thm:2},
that is, we observe that the operations
$(i,j)\to(-i,j)$ and $(i,j)\to(i,-j)$ leave the summand invariant.
However, a notable difference here is that the interchange of
summation indices $(i,j)\to(j,i)$ does not leave the summand
invariant. Consequently, here we see that
{\allowdisplaybreaks
\begin{align} \notag
\sum_{i,j}
&\left\vert i^sj^t(j^{2k}-i^{2k})\right\vert
\binom {2n}{n+i}
\binom {2m}{m+j}\\
\notag
&=4\sum_{0\le i\le j}
\alpha(i=0)\,\alpha(j=0)\,
i^sj^t\left( j^{2k}-i^{2k}\right)
\binom {2n}{n+i}
\binom {2m}{m+j}\\
\notag
&\kern2cm
+4\sum_{0\le i\le j}
\alpha(i=0)\,\alpha(j=0)\,
i^tj^s\left( j^{2k}-i^{2k}\right)
\binom {2n}{n+j}
\binom {2m}{m+i}\\
\notag
&=4\sum_{0\le i\le j}
i^sj^t\left( j^{2k}-i^{2k}\right)
\binom {2n}{n+i}
\binom {2m}{m+j}\\
\notag
&\kern2cm
+4\sum_{0\le i\le j}
i^tj^s\left( j^{2k}-i^{2k}\right)
\binom {2n}{n+j}
\binom {2m}{m+i}\\
&\kern2cm
-2\binom {2n}{n}
\sum_{j=1}^m
0^sj^{t+2k}
\binom {2m}{m+j}
-2\binom {2m}{m}
\sum_{j=1}^n
0^tj^{s+2k}
\binom {2n}{n+j},
\label{eq:7B} 
\end{align}}%
where $\alpha(\mathcal A)$ has the same meaning as in the proof of
Theorem~\ref{thm:2}.
Now one splits the sums into several sums of the form 
\begin{multline*}
\sum_{0\le i\le j}
i^Aj^B
\binom {2n}{n+i}
\binom {2m}{m+j}
\quad \text{and}\quad 
\sum_{0\le i\le j}
i^Aj^B
\binom {2m}{m+i}
\binom {2n}{n+j},
\\
\quad \text{respectively}\quad 
\sum_{j=1}^n
j^{B}
\binom {2n}{n+j}
\quad \text{and}\quad 
\sum_{j=1}^m
j^{B}
\binom {2m}{m+j}.
\end{multline*}
To sums of the second form, we apply Lemma~\ref{lem:5}.
In order to evaluate the sums of the first form, we proceed as in
the proof of Proposition~\ref{prop:0}. 
That is, we apply the expansions \eqref{eq:deci} and
\eqref{eq:decj} (with $n$ replaced by $m$ if appropriate),
and subsequently we use Lemmas~\ref{lem:20}--\ref{lem:22} to
evaluate the sums over $i$ and $j$. Inspection of the result
makes all assertions of the theorem obvious, except for the
claims in Items~(1) and (2) that the polynomial $V^{(2)}_{s,t,k}(n,m)$,
the coefficient of $\binom {2n}n\binom {2m}m$ in \eqref{MNeq:2k1},
vanishes. 

Below we treat Item~(1), that is, the case where $s$ is odd and $t$ is
even. Item~(2) can be handled completely analogously.

After having done the above described manipulations,
we see that, for odd $s$ and even~$t$, the coefficient of
$\binom {2n}n\binom {2m}m$ in the expression \eqref{eq:7B} equals
{\allowdisplaybreaks
\begin{align*}
&\sum_{a=0}^S\sum_{b=0}^{T+k}
c_{a,S}(n)\,c_{b,T+k}(m)\, (2n-2a+1)_{2a}\,(2m-2b+1)_{2b}
\frac {(n-a)(m-b)} {n+m-a-b}\\
&\kern7cm
\times
\binom {2n-2a}{n-a}
\binom {2m-2b}{m-b}
{\binom {2n}{n}}^{-1}
{\binom {2m}{m}}^{-1}
\\
&\kern0cm
-\sum_{a=0}^{S+k}\sum_{b=0}^{T}
c_{a,S+k}(n)\,c_{b,T}(m)\, (2n-2a+1)_{2a}\,(2m-2b+1)_{2b}
\frac {(n-a)(m-b)} {n+m-a-b}
\\
&\kern7cm
\times
\binom {2n-2a}{n-a}
\binom {2m-2b}{m-b}
{\binom {2n}{n}}^{-1}
{\binom {2m}{m}}^{-1}
\\
&\kern0cm
+\sum_{a=0}^T\sum_{b=0}^{S+k}
c_{a,T}(m)\,c_{b,S+k}(n)\, (2m-2a+1)_{2a}\,(2n-2b+1)_{2b}
\frac {(n-b)(n-b+2(m-a))} {n+m-a-b}
\\
&\kern7cm
\times
\binom {2n-2b}{n-b}
\binom {2m-2a}{m-a}
{\binom {2n}{n}}^{-1}
{\binom {2m}{m}}^{-1}
\\
&\kern0cm
-\sum_{a=0}^{T+k}\sum_{b=0}^{S}
c_{a,T+k}(m)\,c_{b,S}(n)\, (2m-2a+1)_{2a}\,(2n-2b+1)_{2b}
\frac {(n-b)(n-b+2(m-a))} {n+m-a-b}
\\
&\kern7cm
\times
\binom {2n-2b}{n-b}
\binom {2m-2a}{m-a}
{\binom {2n}{n}}^{-1}
{\binom {2m}{m}}^{-1}
\\
&\kern0cm
-0^t
\sum_{b=0} ^{S+k} c_{b,S+k}(n)\,(n-b)_{b+1}\,(n-b+1)_b.
\end{align*}}%
In the last two double sums above, we interchange the summation indices
$a$ and~$b$. Then the first and fourth double sum can be combined into
one double sum, as well as the second and third double sum. 
Thus, the above expression simplifies to
{\allowdisplaybreaks
\begin{align*}
&-\sum_{a=0}^S\sum_{b=0}^{T+k}
c_{a,S}(n)\,c_{b,T+k}(m)\, (2n-2a+1)_{2a}\,(2m-2b+1)_{2b}
(n-a)\\
&\kern7cm
\times
\binom {2n-2a}{n-a}
\binom {2m-2b}{m-b}
{\binom {2n}{n}}^{-1}
{\binom {2m}{m}}^{-1}
\\
&\kern0cm
+\sum_{a=0}^{S+k}\sum_{b=0}^{T}
c_{a,S+k}(n)\,c_{b,T}(m)\, (2n-2a+1)_{2a}\,(2m-2b+1)_{2b}
(n-a)
\\
&\kern7cm
\times
\binom {2n-2a}{n-a}
\binom {2m-2b}{m-b}
{\binom {2n}{n}}^{-1}
{\binom {2m}{m}}^{-1}
\\
&\kern0cm
-0^t
\sum_{b=0} ^{S+k} c_{b,S+k}(n)\,(n-b)_{b+1}\,(n-b+1)_b.
\end{align*}}%
In both double sums, the sum over $b$ can be evaluated by means of
\eqref{eq:0^{2S}}. This leads us to the expression
\begin{multline*}
-0^{2T+2k}\sum_{a=0}^S
c_{a,S}(n)\, (n-a)_{a+1}\,(n-a+1)_{a}
+0^{2T}\sum_{a=0}^{S+k}
c_{a,S+k}(n)\, (n-a)_{a+1}\,(n-a+1)_{a}
\\
-0^t
\sum_{b=0} ^{S+k} c_{b,S+k}(n)\,(n-b)_{b+1}\,(n-b+1)_b,
\end{multline*}
which visibly vanishes due to our assumptions that $k>0$
and that $t$ is even.
\end{proof}

We now turn to our results for double sums of the form
\eqref{eq:sumallg2} with odd~$k$ (and $\beta=1$). 
We first state our results for $m=n$
and immediately thereafter the one we obtain for generic $n$ and $m$
in the case where $s$ and $t$ are both odd. We then indicate 
the proofs of both theorems.

\begin{theorem} \label{thm:3}
Let $s,t,k$ and $n$ be non-negative integers. 

If $s$ and $t$ are not both odd, then
\begin{multline} \label{eq:2k-11} 
\sum_{-n\le i,j\le n}\left\vert i^sj^t(j^{2k+1}-i^{2k+1})\right\vert
\binom {2n}{n+i}
\binom {2n}{n+j}\\
=\frac {X^{(1)}_{s,t,k}(n)}{(4n-1)(4n-3)\cdots(4n-2S-2T-2k+1)}{\binom {4n}{2n}}
\kern2cm\\
+
\frac {X^{(2)}_{s,t,k}(n)}{(2n-1)(2n-3)\cdots
(2n-2\cl{(S+T+k)/2}+1)}{\binom {2n}n}^2\\
+X^{(3)}_{s,t,k}(n)\cdot 4^n\binom {2n}n
+X^{(4)}_{s,t,k}(n)\cdot 16^n,
\end{multline}
where the $X^{(i)}_{s,t,k}(n)$, $i=1,2,3,4$, are polynomials in $n$,
$S=\fl{s/2}$ and $T=\fl{t/2}$.

More specifically,

\begin{enumerate} 
\item if $s$ and $t$ are even, then, as polynomials in $n$,
$X^{(1)}_{s,t,k}(n)$ is of degree at most $3S+3T+3k$,
and $X^{(2)}_{s,t,k}(n)$,
$X^{(3)}_{s,t,k}(n)$,
and $X^{(4)}_{s,t,k}(n)$ are identically zero;
\item if $s$ is odd and $t$ is even, then, as polynomials in $n$,
$X^{(2)}_{s,t,k}(n)$ is of degree at most 
$2S+2T+2k+1+\cl{(S+T+k)/2}$,
$X^{(4)}_{s,t,k}(n)$ is of degree at most $2S+2T+2k+1$,
and $X^{(1)}_{s,t,k}(n)$
and $X^{(3)}_{s,t,k}(n)$ are identically zero;
\item if $s$ is even and $t$ is odd, then, as polynomials in $n$,
$X^{(2)}_{s,t,k}(n)$ is of degree at most 
$2S+2T+2k+1+\cl{(S+T+k)/2}$,
$X^{(4)}_{s,t,k}(n)$ is of degree at most $2S+2T+2k+1$,
and $X^{(1)}_{s,t,k}(n)$
and $X^{(3)}_{s,t,k}(n)$ are identically zero.
\end{enumerate}

If $s$ and $t$ are odd, then
\begin{multline} \label{eq:2k-11B} 
\sum_{-n\le i,j\le n}\left\vert i^sj^t(j^{2k+1}-i^{2k+1})\right\vert
\binom {2n}{n+i}
\binom {2n}{n+j}\\
=\frac {X^{(1)}_{s,t,k}(n)}{(4n-1)(4n-3)\cdots(4n-2S-2T-2k-1)}{\binom {4n}{2n}}
+X^{(3)}_{s,t,k}(n)\cdot 4^n\binom {2n}n,
\end{multline}
where $S=\fl{s/2}$ and $T=\fl{t/2}$, and, as polynomials in $n$,
$X^{(1)}_{s,t,k}(n)$ is of degree at most $3S+3T+3k+2$,
and $X^{(3)}_{s,t,k}(n)$ is of degree at most $2S+2T+2k+2$.
\end{theorem}

\begin{theorem} \label{thm:3B}
Let $s,t,k$ and $n,m$ be non-negative integers. 
If $s$ and $t$ are both odd, then
\begin{multline} \label{eq:10B}  
\sum_{i,j}\left\vert i^sj^t(j^{2k+1}-i^{2k+1})\right\vert
\binom {2n}{n+i}
\binom {2m}{m+j}\\
=\frac {Y^{(1)}_{s,t,k}(n,m)}{(2n+2m-1)(2n+2m-3)\cdots(2n+2m-2S-2T-2k-1)}
{\binom {2n+2m}{n+m}}
\\
+Y^{(3)}_{s,t,k}(n,m)\cdot 4^m\binom {2n}n
+Y^{(4)}_{s,t,k}(n,m)\cdot 4^n\binom {2m}m,
\end{multline}
where $S=\fl{s/2}$ and $T=\fl{t/2}$, and, as  polynomials in $n$ and~$m$,
$Y^{(1)}_{s,t,k}(n,m)$ is of degree at most $3S+3T+3k+3$,
and $Y^{(3)}_{s,t,k}(n,m)$ and $Y^{(4)}_{s,t,k}(n,m)$ are of degree at most 
$2S+2T+2k+1$.
\end{theorem}

\begin{remarkno}
From the proof below it is obvious that 
also here (cf.\ Remark~\ref{rem:expl}) explicit formulae for the polynomials
$X^{(i)}_{s,t,k}(n)$ and $Y^{(i)}_{s,t,k}(n,m)$, $i=1,2,3,4$, 
exist in terms of coefficients
$c_{a,A}(n)$ and $c_{b,B}(m)$, for various specific choices of $A$ and~$B$,
with Lemma~\ref{lem:6} providing an explicit formula for these coefficients.
\end{remarkno}

\begin{proof}[Proof of Theorems \ref{thm:3} and \ref{thm:3B}]
We use the operations $(i,j)\to(-i,j)$ and
$(i,j)\to(i,-j)$ (but not $(i,j)\to(j,i)$).
What we get is (for the proof of Theorem~\ref{thm:3} we have to
assume that $m=n$)
{\allowdisplaybreaks
\begin{align}
\notag
\sum_{i,j}&\left\vert i^sj^t(j^{2k+1}-i^{2k+1})\right\vert
\binom {2n}{n+i}
\binom {2m}{m+j}\\
\notag
&=
\frac {1} {2}
\sum_{i,j}\Big(
\left\vert i^sj^t(j^{2k+1}-i^{2k+1})\right\vert
+\left\vert i^sj^t(j^{2k+1}+i^{2k+1})\right\vert
\Big)
\binom {2n}{n+i}
\binom {2m}{m+j}\\
\notag
&=
2
\sum_{0\le i,j}\al(i=0)\,\al(j=0)
\Big(
\left\vert i^sj^t(j^{2k+1}-i^{2k+1})\right\vert\\
\notag
&\kern5cm
+\left\vert i^sj^t(j^{2k+1}+i^{2k+1})\right\vert
\Big)
\binom {2n}{n+i}
\binom {2m}{m+j}\\
\notag
&=
2
\sum_{0\le i\le j}\Big(
\left\vert i^sj^t(j^{2k+1}-i^{2k+1})\right\vert
+\left\vert i^sj^t(j^{2k+1}+i^{2k+1})\right\vert
\Big)
\binom {2n}{n+i}
\binom {2m}{m+j}\\
\notag
&\kern1cm
+2
\sum_{0\le i< j}\Big(
\left\vert i^tj^s(j^{2k+1}-i^{2k+1})\right\vert
+\left\vert i^tj^s(j^{2k+1}+i^{2k+1})\right\vert
\Big)
\binom {2n}{n+j}
\binom {2m}{m+i}\\
\notag
&\kern1cm
-2
\binom {2n}{n}0^s
\sum_{0\le j}
j^{t+2k+1}
\binom {2m}{m+j}
-2
\binom {2m}{m}0^t
\sum_{0\le i}
i^{s+2k+1}
\binom {2n}{n+i}\\
\notag
&=
4
\sum_{0\le i\le j}
i^sj^{t+2k+1}
\binom {2n}{n+i}
\binom {2m}{m+j}
+4
\sum_{0\le i\le j}
i^tj^{s+2k+1}
\binom {2n}{n+j}
\binom {2m}{m+i}\\
\notag
&\kern1cm
-2
\binom {2n}{n}0^s
\sum_{0\le j}
j^{t+2k+1}
\binom {2m}{m+j}
-2
\binom {2m}{m}0^t
\sum_{0\le i}
i^{s+2k+1}
\binom {2n}{n+i}\\
&\kern1cm
-4
\sum_{j\ge1}
j^{s+t+2k+1}
\binom {2n}{n+j}
\binom {2m}{m+j},
\label{eq:7C}
\end{align}}%
where $\alpha(\mathcal A)$ has the same meaning as in the proof of
Theorem~\ref{thm:2}.
To the single sums over~$i$ and over~$j$, 
we apply Lemmas~\ref{lem:5} and \ref{lem:4}.
In order to evaluate the sums over $0\le i\le j$, we proceed as in
the proof of Proposition~\ref{prop:0}. 
That is, we apply the expansions \eqref{eq:deci} and
\eqref{eq:decj} (with $n$ replaced by $m$ if appropriate),
and subsequently we use Lemmas~\ref{lem:10}--\ref{lem:22} to
evaluate the sums over $0\le i\le j$. Inspection of the result
makes all assertions of the theorem obvious, except for the
claims of the vanishing of the polynomial $X^{(2)}_{s,t,k}(n)$ 
in Theorem~\ref{thm:3}, Item~(1),
of the vanishing of the polynomial $X^{(1)}_{s,t,k}(n)$ 
in Theorem~\ref{thm:3}, Items~(2) and~(3),
and of the claim that the coefficient of ${\binom {2n}n}^2$
in Theorem~\ref{thm:3}, right-hand side of
\eqref{eq:2k-11B}, vanishes, as well as 
the coefficient of $\binom {2n}n\binom {2m}m$
in Theorem~\ref{thm:3B}, right-hand side of \eqref{eq:10B}.

Below we treat the last case, that is, the case of generic $n$ and~$m$ 
where $s$ and $t$ are both odd.
The other claims can be handled completely analogously.

Following the above described procedure, using \eqref{eq:aux7b} with
$h=S+T+k+1$ and $k=0$ for the evaluation of the sum over $j$ in the
last line of \eqref{eq:7C}, we obtain from
Lemma~\ref{lem:20} that the coefficient of $\binom {2n}n\binom {2m}m$
in the expression \eqref{eq:7C} equals
\begin{multline*} 
4\sum_{a=0} ^{S}
\sum_{b=0} ^{T+k+1}
c_{a,S}(n)\,c_{b,T+k+1}(m) \,(n-a+1)_{a}^2\,(m-b+1)_{b}^2\,
\frac {(n-a)(m-b)} {4(n+m-a-b)}
\\
+
4\sum_{b=0} ^{T}
\sum_{a=0} ^{S+k+1}
c_{b,T}(m)\,c_{a,S+k+1}(n) \,(n-a+1)_{a}^2\,(m-b+1)_{b}^2\,
\frac {(n-a)(m-b)} {4(n+m-a-b)}\\
-4\sum_{b=0} ^{S+T+k+1}c_{b,S+T+k+1}(n) \,(n-b)_{b+1}\,(n-b+1)_{b}
\frac {m} {2(n+m-b)}.
\end{multline*} 
If we now use \eqref{eq:aux7b} with
$(S,T+k+1)$, $(S+k+1,T)$, and $(S+T+k+1,0)$ in place of $(h,k)$,
we see that the above expression vanishes.
This establishes the assertion about the ``non-appearance" of
the term $\binom {2n}n\binom {2m}m$ in
Theorem~\ref{thm:3B}, and thus also the assertion about 
the ``non-appearance" of ${\binom {2n}n}^2$ in Eq.~\eqref{eq:2k-11B} of
Theorem~\ref{thm:3}.
\end{proof}

\begin{remark} \label{rem:1}
(1) It is obvious from the proofs of Theorems~\ref{thm:2}--\ref{thm:3B} 
that we could deduce analogous theorems for the more
general sums \eqref{eq:sumallg} and \eqref{eq:sumallg2}. 
We omit this here for the sake of brevity, but provide an example
of such an evaluation in \eqref{eq:beta=3} below.

\medskip
(2)
Theorems~\ref{thm:2}--\ref{thm:3B} imply an obvious algorithm to
evaluate a sum of the form \eqref{eq:sumallg} or \eqref{eq:sumallg2} 
for any given
$s,t,k$ and $\beta=1$. (Again, an extension to arbitrary~$\beta$
would be possible.) Namely, addressing the case of odd~$k$ and $m=n$,
one makes an indeterminate Ansatz for the
polynomials
$X^{(1)}_{s,t}(n),X^{(2)}_{s,t}(n),X^{(3)}_{s,t}(n),X^{(4)}_{s,t}(n)$
in Theorem~\ref{thm:3}, one evaluates the sum on the left-hand side
of \eqref{eq:2k-11} for $n=S+T+k,\dots,N+S+T+k$, where $N$ is the number of
indeterminates involved in the Ansatz, giving rise to a system of
$N+1$ linear equations for the $N$ indeterminates. One solves the
system and substitutes the solutions on the right-hand side of
\eqref{eq:2k-11}. 

In this manner, we can establish any of the proved or
conjectured double sum evaluations in \cite{BOOPAA}. For example, we obtain
{\allowdisplaybreaks
\begin{align}
\sum_{-n\le i,j\le n}\left\vert j^3-i^3\right\vert
\binom {2n}{n+i}
\binom {2n}{n+j}
&=
\frac {4n^2(5n-2)} {4n-1}\binom {4n-1}{2n-1},\\
\sum_{-n\le i,j\le n}\left\vert j^5-i^5\right\vert
\binom {2n}{n+i}
\binom {2n}{n+j}
&=
\frac {8n^2(43n^3-70n^2+36n-6)} {(4n-2)(4n-3)}\binom {4n-2}{2n-2},\\
\label{eq:112}
\sum_{i,j}\left\vert ij(j^2-i^2)\right\vert
\binom {2n}{n+i}
\binom {2m}{m+j}
&=
\frac {mn(n^2-n+m^2-m)} {n+m-1}{\binom {2n}{n}}{\binom {2m}{m}},\\
\sum_{i,j}\left\vert i^3j^3(j^2-i^2)\right\vert
\binom {2n}{n+i}
\binom {2m}{m+j}
&=
\frac {2n^2m^2P_1(n,m)} {(n+m-1)(n+m-2)(n+m-3)}
\notag\\
\label{eq:332}
&\kern4.5cm
\times{\binom {2n}{n}}
{\binom {2m}{m}},\\
\sum_{-n\le i,j\le n}\left\vert j^7-i^7\right\vert
\binom {2n}{n+i}
\binom {2n}{n+j}
&=
\frac {16n^2P_2(n)} {(4n-3)(4n-4)(4n-5)}\binom {4n-3}{2n-3},
\end{align}}%
where
\begin{multline*}
P_1(n,m)=n^4+2 n^3 m-6 n^3-6 n^2
   m+11 n^2+2 n m^3-6 n m^2+12 n m\\
-10
   n+m^4-6 m^3+11 m^2-10
   m+4
\end{multline*}
and
$$
P_2(n)= 531n^5 - 1960n^4 + 2800n^3 - 1952n^2 + 668n - 90.
$$
These identities (with $m=n$ for \eqref{eq:112} and \eqref{eq:332}) 
establish the conjectured identities 
(5.7)--(5.9), (5.12), (5.14) from \cite{BOOPAA}.
However, our machinery also yields
\begin{multline}
\sum_{-n\le i,j\le n}\left\vert i^4j^3(j^5-i^5)\right\vert
\binom {2n}{n+i}
\binom {2n}{n+j}\\
=
\frac{n^4\left(414 n^6-2968 n^5+8332 n^4-11853
   n^3+9105 n^2-3592 n+565\right)
   }{2 (2 n-5) (2 n-3) (2
   n-1)}\binom{2 n}{n}^2\\
+\frac {1} {128}n^2 (3 n-1) \left(105
   n^3-210 n^2+147 n-34\right) 16^n
\end{multline}
or
\begin{multline} \label{eq:beta=3}
\sum_{-n\le i,j\le n}\left\vert ij(j^3-i^3)^3\right\vert
\binom {2n}{n+i}
\binom {2n}{n+j}\\
=
\frac {1} {16} n^2 \left(1377 n^4-3870 n^3+4503
   n^2-2442 n+496\right)4^n \binom{2
   n}{n}\\
-\frac{4 n^3 P_3(n) }
   {(4 n-7) (4 n-5) (4 n-3) (4 n-1)}\binom{4 n}{2
   n},
\end{multline}
where
$$
P_3(n)=1917 n^7-11160
   n^6+26439 n^5-33189 n^4+23945 n^3-9951
   n^2+2206 n-201,
$$
for example.
Obviously, one could also use the summation tools ddescribed in
Section~\ref{sec:Sigma} to simplify the left-hand sides to their
right-hand sides.

\medskip
(3) In case the reader wonders what would happen if, instead of
double sums of the form \eqref{eq:sumallg2}, we would consider
double sums of the form
\begin{equation} \label{eq:2n+12m+1} 
\sum_{i,j}  
|i^sj^t(i^k-j^k)^\beta| \binom {2n+1} {n+i} \binom {2m+1} {m+j}
\end{equation}
or mixed sums
\begin{equation} \label{eq:2n+12m} 
\sum_{i,j}  
|i^sj^t(i^k-j^k)^\beta| \binom {2n+1} {n+i} \binom {2m} {m+j},
\end{equation}
we point out that
$$
\binom {2n+1}{n+i}=\frac {n+i+1} {2n+2}\binom {2(n+1)}{n+1+i}
=\frac {1} {2}\binom {2(n+1)}{n+1+i}+
\frac {i} {2n+2}\binom {2(n+1)}{n+1+i},
$$
and thus double sums of the form \eqref{eq:2n+12m+1} or
\eqref{eq:2n+12m} can be written as a linear combination of our
familiar double sums \eqref{eq:sumallg2}.
\end{remark}

\section{An inequality for a binomial double sum}
\label{sec:Ungl}

In this final section, we establish Conjecture~3.1 from \cite{BOOPAA},
which provides a lower bound on sums of the form \eqref{eq:sumallg2}
with $s=t=0$, $k=2$, $\beta=1$.

\begin{theorem} \label{thm:1}
For all non-negative integers $m$ and $n$,
we have
\begin{equation} \label{eq:1} 
\sum_{i,j}\left\vert j^2-i^2\right\vert
\binom {2n}{n+i}
\binom {2m}{m+j}
\ge 2nm\binom {2n}n \binom {2m}m,
\end{equation}
and equality holds if and only if $m=n$.
\end{theorem}

\begin{proof}
Without loss of generality, we assume $m\ge n$.

Using the operations 
$(i,j)\to(-i,j)$
and $(i,j)\to(i,-j)$,
which do not change the summand, we see that \eqref{eq:1} is equivalent
to
\begin{equation} \label{eq:2} 
\sum_{0\le i,j}
\alpha(i=0)\,
\alpha(j=0)\,
\left\vert j^2-i^2\right\vert
\binom {2n}{n+i}
\binom {2m}{m+j}
\ge \frac {nm} {2}\binom {2n}n \binom {2m}m,
\end{equation}
where $\alpha(i=0)$ has the same meaning as in the proof of
Proposition~\ref{thm:2}.
By Lemma~\ref{lem:1}, we see that the claim would be established
if we were able to show that
\begin{equation} \label{eq:4} 
\sum_{0\le i<j}
\alpha(i=0)\left(
\binom {2n}{n+i}
\binom {2m-2}{m+j-1}
-
\binom {2n-2}{n+j-1}
\binom {2m}{m+i}
\right)\end{equation}
is non-negative, with equality holding only if $m=n$.
Indeed, Lemma~\ref{lem:3} says that these two last assertions hold
even for each summand in \eqref{eq:4} individually.
(It is at this point that our assumption $m\ge n$ comes into play.) 
This completes the proof of the theorem.
\end{proof}

\begin{lemma} \label{lem:1}
For all non-negative integers $m$ and $n$, we have
\begin{align} 
\notag
&\sum_{0\le i,j}
\alpha(i=0)\,
\alpha(j=0)\,
\left\vert j^2-i^2\right\vert
\binom {2n}{n+i}
\binom {2m}{m+j}\\
\notag
&\kern.5cm
=\frac {nm} {2}
\binom {2n}{n}
\binom {2m}{m}\\
&\kern1.3cm
+2(m-n)
\sum_{0\le i<j}
\alpha(i=0)\left(
\binom {2n}{n+i}
\binom {2m-2}{m+j-1}
-
\binom {2n-2}{n+j-1}
\binom {2m}{m+i}
\right).
\label{eq:5}
\end{align}
\end{lemma}

\begin{proof}
We write
$$
j^2-i^2=(n^2-i^2)-(m^2-j^2)+(m^2-n^2)
$$
and decompose the sum on the left-hand side of \eqref{eq:5}
into two parts according to whether
$i<j$ or $i>j$.
Thereby, the sum on the left-hand side of \eqref{eq:5} becomes
\begin{multline} \label{eq:3}
(2n-1)_2\sum_{0\le i<j}
\alpha(i=0)\,
\binom {2n-2}{n+i-1}
\binom {2m}{m+j}\\
-(2m-1)_2\sum_{0\le i<j}
\alpha(i=0)\,
\binom {2n}{n+i}
\binom {2m-2}{m+j-1}
\\
-(2n-1)_2\sum_{0\le j<i}
\alpha(j=0)\,
\binom {2n-2}{n+i-1}
\binom {2m}{m+j}\kern3cm\\
+(2m-1)_2\sum_{0\le j<i}
\alpha(j=0)\,
\binom {2n}{n+i}
\binom {2m-2}{m+j-1}
\\
+(m^2-n^2)\sum_{0\le i<j}
\alpha(i=0)\left(
\binom {2n}{n+i}
\binom {2m}{m+j}
-
\binom {2n}{n+j}
\binom {2m}{m+i}\right).
\end{multline}

We next show how to evaluate the first two (double) sums in \eqref{eq:3}.
In the first line of \eqref{eq:3}, we use the decomposition
\begin{equation} \label{eq:bindecomp} 
\binom {2m}{m+j}=
\binom {2m-2}{m+j}
+2\binom {2m-2}{m+j-1}
+\binom {2m-2}{m+j-2},
\end{equation}
while in the second line we use the same decomposition with $m$
replaced by $n$ and $j$
by~$i$. This leads to
\begin{align*}
&(2n-1)_2\sum_{0\le i<j}
\alpha(i=0)\,
\binom {2n-2}{n+i-1}
\binom {2m}{m+j}\\
&\kern1.5cm
-(2m-1)_2\sum_{0\le i<j}
\alpha(i=0)\,
\binom {2n}{n+i}
\binom {2m-2}{m+j-1}
\\
&\kern.5cm
=(2n-1)_2\sum_{0\le i<j}
\alpha(i=0)\,
\binom {2n-2}{n+i-1}
\binom {2m-2}{m+j}\\
&\kern1.5cm
+(2n-1)_2\sum_{0\le i<j}
\alpha(i=0)\,
\binom {2n-2}{n+i-1}
\binom {2m-2}{m+j-2}\\
&\kern1.5cm
-(2n-1)_2\sum_{0\le i<j}
\alpha(i=0)\,
\binom {2n-2}{n+i}
\binom {2m-2}{m+j-1}\\
&\kern1.5cm
-(2n-1)_2\sum_{0\le i<j}
\alpha(i=0)\,
\binom {2n-2}{n+i-2}
\binom {2m-2}{m+j-1}\\
&\kern1.5cm
+\big((2n-1)_2-(2m-1)_2\big)\sum_{0\le i<j}
\alpha(i=0)\,
\binom {2n}{n+i}
\binom {2m-2}{m+j-1}.
\end{align*}
By a simultaneous shift of $i$ and $j$ by one, one sees that the first
and fourth sum on the right-hand side cancel each other largely,
and the same is true for the second and the third sum.
Thus, we have
{\allowdisplaybreaks
\begin{align*} 
&(2n-1)_2\sum_{0\le i<j}
\alpha(i=0)\,
\binom {2n-2}{n+i-1}
\binom {2m}{m+j}\\
&\kern1.5cm
-(2m-1)_2\sum_{0\le i<j}
\alpha(i=0)\,
\binom {2n}{n+i}
\binom {2m-2}{m+j-1}
\\
&\kern.5cm
=-\frac {1} {2}(2n-1)_2\sum_{0<j}
\binom {2n-2}{n-1}
\binom {2m-2}{m+j}\\
&\kern1.5cm
-\frac {1} {2}(2n-1)_2\sum_{0<j}
\binom {2n-2}{n-2}
\binom {2m-2}{m+j-1}\\
&\kern1.5cm
+\frac {1} {2}(2n-1)_2\sum_{0<j}
\binom {2n-2}{n-1}
\binom {2m-2}{m+j-2}\\
&\kern1.5cm
+\frac {1} {2}(2n-1)_2\sum_{0<j}
\binom {2n-2}{n}
\binom {2m-2}{m+j-1}\\
&\kern1.5cm
+\big((2n-1)_2-(2m-1)_2\big)\sum_{0\le i<j}
\alpha(i=0)\,
\binom {2n}{n+i}
\binom {2m-2}{m+j-1}.
\end{align*}}%
Here, there is more cancellation: the second and fourth sum on
the right-hand side cancel each other, while the first and third
cancel each other in large parts, with only two terms remaining.
As a result, we obtain
\begin{align*} 
&(2n-1)_2\sum_{0\le i<j}
\alpha(i=0)\,
\binom {2n-2}{n+i-1}
\binom {2m}{m+j}\\
&\kern1.5cm
-(2m-1)_2\sum_{0\le i<j}
\alpha(i=0)\,
\binom {2n}{n+i}
\binom {2m-2}{m+j-1}
\\
&\kern.5cm
=
\frac {1} {2}(2n-1)_2
\binom {2n-2}{n-1}
\binom {2m-1}{m}\\
&\kern1.5cm
+\big((2n-1)_2-(2m-1)_2\big)\sum_{0\le i<j}
\alpha(i=0)\,
\binom {2n}{n+i}
\binom {2m-2}{m+j-1}\\
&\kern.5cm
=
\frac {n^2} {4}
\binom {2n}{n}
\binom {2m}{m}\\
&\kern1.5cm
+\big((2n-1)_2-(2m-1)_2\big)\sum_{0\le i<j}
\alpha(i=0)\,
\binom {2n}{n+i}
\binom {2m-2}{m+j-1}.
\end{align*}
The same calculation, with $n$ and $m$ interchanged, yields
\begin{align*} 
&-(2n-1)_2\sum_{0\le j<i}
\alpha(j=0)\,
\binom {2n-2}{n+i-1}
\binom {2m}{m+j}\\
&\kern1.5cm
+(2m-1)_2\sum_{0\le j<i}
\alpha(j=0)\,
\binom {2n}{n+i}
\binom {2m-2}{m+j-1}\\
&\kern.5cm
=
\frac {m^2} {4}
\binom {2n}{n}
\binom {2m}{m}\\
&\kern1.5cm
+\big((2m-1)_2-(2n-1)_2\big)\sum_{0\le i<j}
\alpha(i=0)\,
\binom {2m}{m+i}
\binom {2n-2}{n+j-1}.
\end{align*}
If we put everything together, then we have shown that the sum on the
left-hand side of \eqref{eq:5} equals
{\allowdisplaybreaks
\begin{align*}
&\kern.5cm
\frac {n^2+m^2} {4}
\binom {2n}{n}
\binom {2m}{m}\\
&\kern1.5cm
+\big(4(m^2-n^2)-2(m-n)\big)\\
&\kern2cm
\times
\sum_{0\le i<j}
\alpha(i=0)\left(
\binom {2n-2}{n+j-1}
\binom {2m}{m+i}
-
\binom {2n}{n+i}
\binom {2m-2}{m+j-1}\right)\\
&\kern1.5cm
+(m^2-n^2)\sum_{0\le i<j}
\alpha(i=0)\left(
\binom {2n}{n+i}
\binom {2m}{m+j}
-
\binom {2n}{n+j}
\binom {2m}{m+i}\right).
\end{align*}}%
If we finally use Lemma~\ref{lem:2} in this expression, then the
result is the right-hand side of~\eqref{eq:5}.
\end{proof}

\begin{lemma} \label{lem:2}
For all non-negative integers $m$ and $n$, we have
\begin{multline} \label{eq:lem2}
4
\sum_{0\le i<j}
\alpha(i=0)\left(
\binom {2n-2}{n+j-1}
\binom {2m}{m+i}
-
\binom {2n}{n+i}
\binom {2m-2}{m+j-1}\right)\\
+\sum_{0\le i<j}
\alpha(i=0)\left(
\binom {2n}{n+i}
\binom {2m}{m+j}
-
\binom {2n}{n+j}
\binom {2m}{m+i}\right)\\
=-\frac {m-n} {4(m+n)}\binom {2n}n\binom {2m}m.
\end{multline}
\end{lemma}

\begin{proof}
Using the decomposition \eqref{eq:bindecomp} in the second line
of \eqref{eq:lem2}, we compute
{\allowdisplaybreaks
\begin{align*} 
&
4
\sum_{0\le i<j}
\alpha(i=0)\left(
\binom {2n-2}{n+j-1}
\binom {2m}{m+i}
-
\binom {2n}{n+i}
\binom {2m-2}{m+j-1}\right)\\
&\kern1.5cm
+\sum_{0\le i<j}
\alpha(i=0)\left(
\binom {2n}{n+i}
\binom {2m}{m+j}
-
\binom {2n}{n+j}
\binom {2m}{m+i}\right)\\
&
=
\sum_{0\le i<j}
\alpha(i=0)\bigg(
2\binom {2n-2}{n+j-1}
\binom {2m}{m+i}\\
&\kern3.5cm
-
\binom {2n-2}{n+j}
\binom {2m}{m+i}
-
\binom {2n-2}{n+j-2}
\binom {2m}{m+i}\\
&\kern1.5cm
+\binom {2n}{n+i}
\binom {2m-2}{m+j}
+\binom {2n}{n+i}
\binom {2m-2}{m+j-2}
-
2\binom {2n}{n+i}
\binom {2m-2}{m+j-1}\bigg)\\
&
=
\sum_{0\le i}
\alpha(i=0)\bigg(
\binom {2n-2}{n+i}
\binom {2m}{m+i}
-
\binom {2n-2}{n+i-1}
\binom {2m}{m+i}\\
&\kern3.5cm
+\binom {2n}{n+i}
\binom {2m-2}{m+i-1}
-
\binom {2n}{n+i}
\binom {2m-2}{m+i}\bigg)\\
&
=\frac {m-n} {m+n}
\sum_{0\le i}
\alpha(i=0)
\bigg(
\frac {(2n-2)!\,(2m-2)!\,(4nm-4(i+1)n-4(i+1)m+1)} 
{(n+i)!\,(n-i-1)!\,(m+i)!\,(m-i-1)!}\\
&\kern5cm
-
\frac {(2n-2)!\,(2m-2)!\,(4nm-4in-4im+1)} 
{(n+i-1)!\,(n-i)!\,(m+i-1)!\,(m-i)!}
\bigg)\\
&
=\frac {m-n} {m+n}
\bigg(
-\frac {1} {2}
\frac {(2n-2)!\,(2m-2)!\,(4nm-4n-4m+1)} 
{n!\,(n-1)!\,m!\,(m-1)!}\\
&\kern3cm
-\frac {1} {2}
\frac {(2n-2)!\,(2m-2)!\,(4nm+1)} 
{(n-1)!\,n!\,(m-1)!\,m!}
\bigg)\\
&
=-\frac {m-n} {4(m+n)}\binom {2n}n\binom {2m}m,
\end{align*}}%
which is the desired result.\footnote{For the
finding of the telescoping form of the sum over $i$
see Footnote~\ref{foot:1}.} 
\end{proof}

\begin{lemma} \label{lem:3}
For all non-negative integers $m,n,i,j$ with $m\ge n$ and $i<j$, we have
$$
\binom {2n}{n+i}\binom {2m-2}{m+j-1}\ge
\binom {2n-2}{n+j-1}\binom {2m}{m+i} ,
$$
with equality if and only if $m=n$.
\end{lemma}

\begin{proof}
We have
\begin{align*}
\frac {\binom {2n}{n+i}\binom {2m-2}{m+j-1}} 
{\binom {2n-2}{n+j-1}\binom {2m}{m+i}} 
&=
\frac {2n(2n-1)} {2m(2m-1)}
\frac {(m-j+1)(m-j)} {(n-j+1)(n-j)}
\prod _{k=i+1} ^{j-1}
\frac {(n+k)(m-k+1)} 
{(n-k+1)(m+k)}\\
&=
\frac {\left(2+\frac {2j-2}{n-j+1}\right)
\left(2+\frac {2j-1}{n-j}\right)}
{\left(2+\frac {2j-2}{m-j+1}\right)
\left(2+\frac {2j-1}{m-j}\right)}
\prod _{k=i+1} ^{j-1}
\frac {nm+km-(k-1)n-k(k-1)} 
{nm-(k-1)m+kn-k(k-1)}\ge1,
\end{align*}
and visibly equality holds if and only if $m=n$. 
\end{proof}

\section*{Acknowledgements}
The authors thank an anonymus referee for an extremely careful reading
of the original manuscript and for the many suggestions leading to an 
improved presentation.

\end{document}